\documentclass[11pt]{article}

\usepackage[margin=1.1in]{geometry}
\usepackage{amsmath,amssymb,amsthm,mathtools,bbm,bm}
\usepackage{microtype,enumitem,authblk}
\usepackage[colorlinks=true,linkcolor=blue,citecolor=blue,urlcolor=blue]{hyperref}

\newtheorem{theorem}{Theorem}[section]
\newtheorem{proposition}[theorem]{Proposition}
\newtheorem{lemma}[theorem]{Lemma}
\newtheorem{corollary}[theorem]{Corollary}
\newtheorem{remark}[theorem]{Remark}
\newtheorem{defn}[theorem]{Definition}
\numberwithin{equation}{section}

\newcommand{\E}{\mathbb{E}}

\newcommand{\op}{\mathrm{op}}
\newcommand{\HS}{\mathrm{HS}}

\newcommand{\cN}{\mathcal{N}}
\newcommand{\cE}{\mathcal{E}}
\newcommand{\cZ}{\mathcal{Z}}

\newcommand{\dd}{\,\mathrm{d}}

\title{Higher-Order Fluctuations of the Log-Partition Function for Ising Models on Inhomogeneous Random Graphs}

\author[1]{Sanchayan Bhowal}
\author[2]{Somabha Mukherjee}

\affil[1]{Department of Statistics,
Stanford University\\
\texttt{sbhowal@stanford.edu}}

\affil[2]{Department of Statistics and Data Science,
National University of Singapore\\
\texttt{somabha@nus.edu.sg}}

\date{}

\begin{document}
\maketitle

\begin{abstract}
Classical central limit theorems (CLTs) for log-partition functions of Ising models on Erd\H{o}s-R\'enyi random graphs lead to degenerate limits in the high-temperature regime when self-loops are absent. In this paper, we fill this surprising gap in the more general context of Ising models on inhomogeneous random graphs. In particular, we show that if $G(N,W)$ is an inhomogeneous random graph without self-loops generated by a graphon $W$, then, in the high-temperature regime, the log-partition function of the corresponding Ising model satisfies a CLT on the $N$-scale, with a limiting variance that incorporates contributions from cycles of all orders. This behavior contrasts sharply with that of the corresponding model in which self-loops are allowed, where the analogous Gaussian fluctuations occur on the \(\sqrt{N}\)-scale, with their asymptotic variance determined entirely by the diagonal profile of $W$. Our analysis yields, as a byproduct of independent interest, a CLT for the log-determinant of the associated resolvent matrix, established by combining the Lindeberg CLT for triangular arrays with suitable bounds on higher-order trace terms.
\end{abstract}

\section{Introduction}

 A classical model for studying interacting spins and, more broadly, dependence across networked systems, is the Ising model. Originally developed in statistical physics to model ferromagnetism \cite{ising1925}, it has since become a fundamental modeling tool in a wide range of fields, including social-network analysis, image processing, computer vision, pattern recognition, spatial statistics, disease mapping, computational biology, and neural networks \cite{sudipto,geman,green,hopfield,montanari,isingimage7,patrec}. Given a symmetric \(N\times N\) interaction matrix \(M\), the Ising model at inverse temperature \(\beta>0\) is the probability measure on \(\{-1,1\}^N\) defined by

\begin{equation}\label{isingmoddef2}
    \mathbb{P}_{\beta,M}(\boldsymbol{\sigma})
=
\frac{1}{2^N \cZ_N(\beta,M)}
\exp\left(
\frac{\beta}{2N}\boldsymbol{\sigma}^{\top}M\boldsymbol{\sigma}
\right),
\qquad
\boldsymbol{\sigma}\in\{-1,1\}^N,
\end{equation}
where
$$
\cZ_N(\beta,M)
:=
2^{-N}
\sum_{\boldsymbol{\sigma}\in\{-1,1\}^N}
\exp\left(
\frac{\beta}{2N}\boldsymbol{\sigma}^{\top}M\boldsymbol{\sigma}
\right) = \E_{\bm{\sigma}\sim \text{Unif}(\{-1,1\}^N)}\exp\left(\frac{\beta}{2N}\bm{\sigma}^{\top}M\bm{\sigma}\right)
$$
is the partition function, with $\text{Unif}(\{-1,1\}^N)$ denoting the uniform distribution on the hypercube $\{-1,1\}^N$. Understanding the behavior of this model for different types of interaction structure has received widespread attention. 

When the underlying interactions are induced by a random graph, one natural choice is the Erd\H{o}s-R\'enyi model, in which edges appear independently with a common probability \cite{BovierGayrard1993}. In this case, the interaction matrix $M$ appearing in \eqref{isingmoddef2} is the adjacency matrix of the Erd\H{o}s-R\'enyi graph, and the partition function $\cZ_N(\beta,M)$ therefore naturally inherits the randomness of the underlying edge configuration.  This naturally leads to the question of how this randomness is reflected in the second-order asymptotic behavior of the partition function. High-temperature fluctuations of the partition function for the Ising model on Erd\H{o}s-R\'enyi random graphs with self-loops, have been studied in various sparsity regimes by Kabluchko, L\"owe, and Schubert \cite{kabluchko2021partition}. In particular, the following is a direct consequence of their result for the dense regime at high temperatures:
\begin{equation*}
    \sqrt{N}\left(\log \cZ_N(\beta,A_N) - \log \mathbb{E} \cZ_N(\beta,A_N)\right)  \xrightarrow{\mathrm d} \cN\left(0,\frac{\beta^2 p(1-p)}{4}\right)
\end{equation*}
where $A_N$ denotes the adjacency matrix of the Erd\H{o}s-R\'enyi random graph. An inspection of their proof strategy, however, reveals that, in the absence of self-loops, the $\sqrt{N}$-scale limit above is degenerate, rendering the corresponding asymptotic result uninformative. More importantly, this suggests that the fluctuations of $\log \cZ_N(\beta,A_N)$ around its annealed mean $\log \mathbb{E}\cZ_N(\beta,A_N)$ are of smaller order than $N^{-\frac{1}{2}}$.

Since self-loops are uncommon in real-world networks, it is natural to ask, from a practical perspective, how their absence affects the limiting behavior of the associated Ising models.
In this paper, we address this gap by showing that, in the absence of self-loops, the second-order fluctuations of the Ising log-partition function occur on the $N$-scale for a broader class of inhomogeneous random graphs, encompassing the Erd\H{o}s-R\'enyi model as a special case. This phenomenon has also been observed for the \(p\)-spin Sherrington–Kirkpatrick model for even $p>2$, where the versions allowing and excluding coinciding indices are identical at the level of their thermodynamic functions but exhibit high-temperature log-partition function fluctuations on the \(N^{-(p-2)/4}\) and \(N^{-(p-2)/2}\) scales, respectively \cite{bovier2002rem}.

Inhomogeneous random graphs provide a natural framework for modeling networks in which edge probabilities vary across vertex pairs due to latent heterogeneity among the vertices
\cite{bickel2009nonparametric,bickel2011method,crane2018probabilistic,bollobas2007phase,hoff2002latent}.
A widely used class of such models is specified through an underlying graphon -- a symmetric, Borel measurable function
$W:[0,1]^2\to[0,1]$ \cite{Borgs2008,lovasz2012large}.
In this paper, we consider the following dense graphon model.

\begin{defn}
Let $W$ be a Riemann integrable graphon satisfying
$\int_{[0,1]^2}W(x,y)\,\mathrm dx\,\mathrm dy>0$.
The $W$-random graph, denoted by $G(N,W)$, is the random graph on the vertex set
$[N]\coloneq\{1,2,\ldots,N\}$ in which vertices $i$ and $j$ are connected with probability
$
\omega_{ij} := W\left(\frac{i}{N},\frac{j}{N}\right),
$
independently over all $1\leq i< j\leq N$, and does not contain any self-loop.
\end{defn}

For a $W$-random graph $G(N,W)$, let $A_N$ denote its hollow (with diagonal entries $0$) adjacency matrix. $A_N$ is a symmetric random matrix whose upper-triangular entries are independent and satisfy
\(
 A_{ij}\sim\operatorname{Bernoulli}(\omega_{ij})
\) for $i<j.$ The central object of interest in this paper is the logarithm of the partition function \(\cZ_N(\beta,A_N)\) of the Ising model on \(G(N,W)\), corresponding to the choice \(M=A_N\) in \eqref{isingmoddef2}. In this paper, we show that $\log \cZ_N(\beta,A_N)$, centered by its annealed mean $\log\E\cZ_N(\beta,A_N)$, exhibits second-order Gaussian fluctuations on the $N$-scale.

\begin{theorem}\label{thm:main}
    Suppose that $\beta>0$ is such that \(\beta\|T_W\|_{\op}<1\), where $(T_Wf)(x) :=\int_0^1W(x,y)f(y)\dd y$ for any $f \in L^2[0,1]$. Define
    \begin{eqnarray*}
        && U_\beta(x,y) =W(x,y)+\sum_{m=2}^{\infty}\beta^{m-1}\int_{[0,1]^{m-1}}
 W(x,x_1)W(x_1,x_2)\cdots W(x_{m-1},y)
 \dd x_1\cdots\dd x_{m-1},\\
&&\text{and}\qquad \tau_\beta^2 = \tau_{\beta,W}^2 
:=
\frac{\beta^4}{8}
\int_{[0,1]^2}
W(x,y)\bigl(1-W(x,y)\bigr)
\left[
1+2\bigl(U_\beta(x,y)-W(x,y)\bigr)
\right]^2
\dd x\,\dd y.
    \end{eqnarray*}

Then, as \(N\to\infty\), 
    \[
    N\left(
        \log\cZ_N(\beta,A_N)
        -\log\E\cZ_N(\beta,A_N)
    \right)
    \xrightarrow{\mathrm d}
    \cN(0,\tau_\beta^2).
    \]
\end{theorem}

The limiting variance \(\tau_\beta^2\) captures, through \(U_\beta\), contributions from weighted paths of every length connecting \(x\) and \(y\), and hence from the full hierarchy of cycles passing through \(x\) and \(y\). 
Taking $W\equiv p$ in Theorem \ref{thm:main} gives us the following central limit theorem (CLT) for the log-partition function in the dense Erd\H{o}s-R\'enyi Ising model.

\begin{corollary}\label{corerintro}
    Suppose that $0<\beta<\frac{1}{p}$. Then, as $N\rightarrow \infty$,
   $$ N\left(
\log\cZ_N(\beta,A_N)
-\log\E\cZ_N(\beta,A_N)
\right)
\xrightarrow{\mathrm d}
\cN\left(0,\tau_{\beta,p}^2\right)$$
where $$\tau_{\beta,p}^2 := \frac{\beta^4p(1-p)}{8}
\left(
\frac{1+\beta p(2p-1)}{1-\beta p}
\right)^2.$$
\end{corollary}

The study of fluctuations of the free energy and log-partition function has witnessed a remarkable surge of interest in recent years, particularly for spin systems on random graphs.
Fluctuations of the partition function for Ising models on Erd\H{o}s--R\'enyi graphs have been studied in various sparsity regimes by Kabluchko, L\"owe, and Schubert \cite{kabluchko2021partition}.  In the bounded-average-degree setting, Coja-Oghlan et al.\ \cite{cojaoghlan2026isingfree} established both Gaussian and non-Gaussian fluctuations of the Erd\H{o}s--R\'enyi Ising log-partition function at different phases.  The recent work of Prodromidis and Sly \cite{prodromidis2026critical} resolves the corresponding critical Erd\H{o}s--R\'enyi Ising problem and reveals an additional source of randomness coming from counts of paths of slowly diverging length.  Taken together, these results illustrate a recurring phenomenon: second-order fluctuations can be considerably more sensitive than first-order thermodynamic limits to the precise structure of the interaction matrix.  The distinction between the self-loop model \cite{mukherjee2026isinginference} and the hollow model considered here highlights another instance of this phenomenon, where the fluctuation scale itself changes from $\sqrt{N}$ to $N$.

Fluctuations of partition functions and free energies have a long history of study in spin-glass models.
For the Sherrington-Kirkpatrick (SK) model, Gaussian fluctuations in the high-temperature regime were established in the classical works of Aizenman, Lebowitz, and Ruelle \cite{aizenman1987rigorous} and Comets and Neveu \cite{comets1995sk}, with subsequent extensions to $p$-spin and mixed $p$-spin models \cite{bovier2002rem,banerjee2021mixed,bovier2024twoscales}.  For the SK model with weak external fields, Dey and Wu \cite{dey2023weakfield} identified different field-strength subregimes in which the Gaussian fluctuations have variances of different orders. For the SK model with an additional Curie-Weiss interaction, Banerjee \cite{banerjee2020skcw} proved Gaussian fluctuations in the paramagnetic regime, while the recent work of Dey and Kang \cite{dey2026skfi} studied the ferromagnetic regime and the critical window separating the two phases.   Related fluctuation transitions are particularly explicit for spherical spin glasses: In the absence of an external magnetic field, Baik and Lee \cite{baik2016spherical} obtain Gaussian fluctuations at high temperature and Tracy--Widom fluctuations at low temperature, while the regimes of ferromagnetic interactions and non-vanishing external fields have been addressed in \cite{baik2017sphericalferro,baik2021external}. 

The starting point of our proof, and a result of independent interest, is a CLT for \(\log\cZ_N(\beta,A_N)\) relative to its mean-network counterpart \(\log\cZ_N(\beta,\E A_N)\). These fluctuations occur on the same \(N\)-scale but require an additional deterministic bias correction under the mean-network centering. The proof of this mean-network-centered CLT involves comparing the Ising log-partition function directly with the log-determinant of the associated random matrix through a stepwise Rademacher-to-Gaussian replacement argument. This essentially reduces the problem to obtaining a CLT for a random log-determinant. We then expand this log-determinant, after subtracting its mean-network counterpart and applying an additional bias correction, in terms of traces of powers of a transformed version of the centered matrix \(A_N-\E A_N\). The leading linear term in this expansion can be handled using classical CLTs for triangular arrays, while the higher-order trace terms can be shown to converge in probability and first moment to explicit constants. The annealed-centered CLT is then obtained from the bias-corrected mean-network CLT by an exact edgewise moment-generating-function factorization argument.

\renewcommand{\P}{\mathbb{P}}
\renewcommand{\d}{\dd}
\newcommand{\tr}{\operatorname{Tr}}
\newcommand{\bW}{\bm{W}}
\newcommand{\mudet}{\mu_{\det}}
\newcommand{\sbw}{\tau_\beta}
\newtheorem{thm}[theorem]{Theorem}
\newtheorem{cor}[theorem]{Corollary}
\newtheorem{ppn}[theorem]{Proposition}

\subsection{Preliminaries}
We begin by introducing some notation that will be used throughout our analysis:

\begin{enumerate}
    \item  Define the $N\times N$ matrix $\bW_N := ((\omega_{ij}))_{1\le i,j\le N}$, where we set $\omega_{ii}=0$ for all $1\le i \le N$. Also, denote the centered adjacency matrix $A_N - \bW_N$ by $\bar{A}_N$ .

    \item For a matrix $M$, we define the corresponding local field vector as the function:
    $$\bm L(\bm x) = \bm L^{(M)}(\bm x) := \frac{\beta}{N} M^\top \bm x.$$ The $i^\mathrm{th}$ component of $\bm L$ is denoted by $L_i$.

    \item Define the Gaussian partition function:
    $$Z^G(B) :=\E_{\bm{Z}\sim \cN(\bm 0, I_N)}\exp\left(\frac12\bm Z^{\top}B\bm{Z}\right).$$

\item Throughout, the symbols \(\|\cdot\|_{\op}\), \(\|\cdot\|_{\HS}\), and
\(\|\cdot\|_F\) denote the operator, Hilbert-Schmidt, and Frobenius
norms, respectively.

\item  For nonnegative quantities \(a,b\), the relation
\(a\lesssim_\theta b\) means \(a\le C_\theta b\) for a constant $C_\theta$ depending
only on \(\theta\); the subscript is omitted when the constant is absolute.
\end{enumerate}

A fundamental distinction between the annealed-centered asymptotics of the models with and without self-loops lies in the structure of their limiting Gaussian variances. In the self-loop model, the variance established in \cite{mukherjee2026isinginference} is determined entirely by the diagonal profile $x\mapsto W(x,x)$ through the corresponding Bernoulli variance $W(x,x)(1-W(x,x))$. By contrast, in the hollow model, the variance reflects the cumulative effect of multistep connectivity across the graphon, as encoded by the resolvent transform of the graphon integral operator
\begin{equation*}
U_\beta:=T_W(I-\beta T_W)^{-1}.
\end{equation*}
This operator is Hilbert--Schmidt and has kernel
\begin{equation}\label{Useriesuf4}
 U_\beta(x,y) =W(x,y)+\sum_{m=2}^{\infty}\beta^{m-1}W_m(x,y),
\end{equation}
where for \(m\ge2\),
\[
 W_m(x,y):=\int_{[0,1]^{m-1}}
 W(x,x_1)W(x_1,x_2)\cdots W(x_{m-1},y)
 \d x_1\cdots\d x_{m-1}.
\]
Note that for \(m\ge2\),
\[
 |W_m(x,y)|
 \le \|W(x,\cdot)\|_2\,\|T_W\|_{\op}^{m-2}\,
       \|W(\cdot,y)\|_2
 \le \|T_W\|_{\op}^{m-2}.
\]
Consequently, in our high-temperature regime, this series converges uniformly.  

In this context, define the diagonal profiles
\begin{align}\label{hbbarseriesuf4}
 \bar h_\beta(x)
 &:=\sum_{m=2}^{\infty}\beta^mW_m(x,x)\qquad \text{and}\qquad
 h_\beta(x)
 :=\beta^2d_W(x)+\sum_{m=3}^{\infty}\beta^mW_m(x,x),
\end{align}
where $d_W(x) := \int_0^1 W(x,y)\d y.$ If we define \[s_W(x) := \int_0^1W(x,y)(1-W(x,y))\d y = d_W(x)-W_2(x,x),\] then
\(
 h_\beta(x)=\bar h_\beta(x)+\beta^2s_W(x).
\)
We will later show (see Lemma~\ref{lem:riemann-approximation}) that the kernels \(W_m\), the
functions \(d_W,s_W\), and the diagonal profiles
\(x\mapsto W_m(x,x)\), for $m\ge 2$, are all Riemann integrable.  Uniform convergence
of the series \eqref{Useriesuf4} and \eqref{hbbarseriesuf4} therefore shows that \(U_\beta(\cdot,\cdot)\),
\(\bar h_\beta\) and \(h_\beta\) are all Riemann integrable.

\subsection{Sketch of Proof}
The proof of Theorem~\ref{thm:main} proceeds through the following steps.
\begin{enumerate}
    \item \textbf{Rademacher-Gaussian comparison.} We begin by comparing $\log\cZ_N(\beta,A_N)$ to the Gaussian
log-determinant $-\frac12\log\det(I-\beta A_N/N)$ by successively
replacing the Rademacher spin coordinates with Gaussian ones (Proposition \ref{prop:RG-comparison}). The resulting comparison error can be expressed in terms of the diagonal profile of the underlying matrix. Consequently, the difference between $\log\cZ_N(\beta,A_N)$ and its mean-network counterpart $\log\cZ_N(\beta,\bm W_N)$ admits a decomposition into the log-determinant difference $L_N(A_N)-L_N(\bW_N)$ and a remainder involving the corresponding diagonal-profile errors for $A_N$ and $\bW_N$.

\item \textbf{Log-determinant CLT.} Next, we prove a CLT for the resulting log-determinant
difference $L_N(A_N)-L_N(\bW_N)$ (Theorem \ref{thm:logdet}). We expand this
difference as an infinite series of traces of powers of $R_N \bar{A}_N$, which, after an additional bias correction and appropriate rescaling, yields a leading linear term obeying a standard triangular-array
CLT with limiting variance governed by the transformed resolvent $U_\beta$. All the other higher-order trace terms have explicit limits, which contribute to the mean of the limiting Gaussian distribution.

\item \textbf{Tracking diagonal-profile errors.} The cumulative diagonal-profile errors are then shown to have an explicit probability limit (Proposition \ref{prop:diagonal-corrections}), which cancels out the asymptotic mean in the log-determinant CLT introduced by the higher-order trace terms, and establishes the mean-network-centered CLT.

\item \textbf{From mean-network to annealed-centered CLT.} In this final step, we use an edgewise moment-generating-function factorization to express the annealed log-partition function as the sum of its mean-network counterpart, an explicit bias term, and a smaller-order error. This bias term exactly cancels the corresponding bias appearing in the centering of the mean-network-centered CLT, thereby establishing the annealed-centered CLT with a bias-free centering.
\end{enumerate}

\begin{remark}
For the Ising model \eqref{isingmoddef2} with self-loops, it was shown in
\cite{mukherjee2026isinginference}, under the regularity and high-temperature
assumptions therein, that
\begin{equation}\label{prevpp4644}
    \sqrt{N}\left(
        \log\cZ_N(\beta,A_N)
        -\log\E\cZ_N(\beta,A_N)
    \right)
    \xrightarrow{\mathrm d}
    \cN\left(
        0,\frac{\beta^2}{4}
        \int_0^1 W(x,x)\bigl(1-W(x,x)\bigr)\dd x
    \right),
\end{equation}
where $A_N$ denotes the adjacency matrix of the graph with self-loops.
The model without self-loops considered here, can be obtained by setting
$W(x,x)=0$ for all $x\in[0,1]$. In this case, the limiting distribution in
\eqref{prevpp4644} is degenerate at zero. Thus, although this result
suggests that the fluctuations are of smaller order than $N^{-\frac{1}{2}}$,
it does not identify their precise scale or limiting distribution.
Theorem~\ref{thm:main} addresses this issue by establishing a CLT on
the $N$-scale, showing that the correct fluctuation scale is $N^{-1}$
whenever $\tau_\beta^2>0$.
\end{remark}

\subsection{AI Declaration}
GPT-5.6 Sol was used to assist with and verify computations. The authors assume responsibility for all content.

\section{Proof of Theorem~\ref{thm:main}}

In this section, we prove Theorem \ref{thm:main}. The following proposition gives a CLT for the log-partition function around its mean-network counterpart, a result that is central to our analysis and of interest in its own right.

\begin{proposition}\label{thm:mainprop}
Suppose that $\beta>0$ is such that \(\beta\|T_W\|_{\op}<1\). Then,
    \[
    N\left(
        \log\cZ_N(\beta,A_N)
        -\log\cZ_N(\beta,\bW_N)
        -\frac{\beta^2}{2N^2}\sum_{1\le i<j\le N}\omega_{ij}(1-\omega_{ij})
    \right)
    \xrightarrow{\mathrm d}
    \cN(0,\tau_\beta^2).
    \]
\end{proposition}

\begin{proof}[Proof of Proposition \ref{thm:mainprop}]
    {To begin with, define:
    \begin{equation}\label{RNdef7224}
        R_N:=\left(I-\frac{\beta\bW_N}{N}\right)^{-1}
    \end{equation}
    Let \(\rho\) be such that $\beta\|T_W\|_{\op}<\rho<1$. Define the following  event:
\begin{equation}\label{epNrhodef}
    \cE_N^\rho:=\left\{\beta\left\|\frac{A_N}{N}\right\|_{\op}\le\rho\right\}.
\end{equation}
By Lemma \ref{lem:high-temperature} (ii), we know that $\mathbb{P}(\cE_N^\rho) \rightarrow 1$ as $N\rightarrow \infty$. Set 
\begin{equation}\label{rnhatdef7842}
 \widehat R_N
 :=
 \begin{cases}
 \left(I-\dfrac{\beta A_N}{N}\right)^{-1} &\text{on }\cE_N^\rho,\\[1ex]
 I &\text{on }(\cE_N^\rho)^c.
 \end{cases}
\end{equation}

This truncation keeps \(\widehat R_N\) well defined
even when \(I-\beta A_N/N\) fails to be invertible. Next, whenever
\(I-\beta M/N\) is positive definite, write
\(L_N(M):=\log\det(I-\beta M/N)\); for \(M=\bW_N\) this is used for all
sufficiently large \(N\), while for \(M=A_N\) we use it on \(\cE_N^\rho\)
and set \(L_N(A_N) =0\) on \((\cE_N^\rho)^c\), so
that \(L_N(A_N)\) is well defined on the whole probability space.} By
Lemma~\ref{lem:high-temperature} (i), for all sufficiently large \(N\), we have:
\[
\beta\left\|\frac{\bW_N}{N}\right\|_{\op}\le\rho.
\]
On the event \(\cE_N^\rho\),  Proposition~\ref{prop:RG-comparison}
applies to \(B = \beta M/N\) for both
\(M=A_N\) and \(M=\bW_N\), with
\(\delta=1-\rho\). Also, note that
\(\log Z^G(\beta M/N)=-L_N(M)/2\) on \(\cE_N^\rho\) for both these choices of $M$.  Applying the
proposition separately to \(A_N\) and \(\bW_N\), we thus have on \(\cE_N^\rho\):
\begin{align*}
&\log\cZ_N(\beta,A_N)-\log\cZ_N(\beta,\bW_N)
\\
&\quad=
\left[
-\frac12L_N(A_N)
-\frac14\sum_{i=1}^N\{\widehat R_{N,ii}-1\}^2
\right]
-
\left[
-\frac12L_N(\bW_N)
-\frac14\sum_{i=1}^N(R_{N,ii}-1)^2
\right]
+O_{\P}(N^{-2})
\\
&\quad=
-\frac12[L_N(A_N)-L_N(\bW_N)]
-\frac14\sum_{i=1}^N
\left[
(\widehat R_{N,ii}-1)^2
-(R_{N,ii}-1)^2
\right]
+O_{\P}(N^{-2}).
\end{align*}
Hence, on $\cE_N^\rho$, we have:
\begin{equation}\label{epfdl1}
\begin{aligned}
&N\Bigg(
\log\cZ_N(\beta,A_N)-\log\cZ_N(\beta,\bW_N)-\frac{\beta^2}{2N^2}\sum_{i<j}\omega_{ij}(1-\omega_{ij})
\Bigg)\\
&=-\frac N2 \left(L_N(A_N)-L_N(\bW_N)+\frac{\beta^2}{N^2}\sum_{i<j}\omega_{ij}(1-\omega_{ij})\right)\\
&\qquad-\frac N4\sum_{i=1}^N
\left[(\widehat R_{N,ii}-1)^2-(R_{N,ii}-1)^2\right]+o_{\P}(1).
\end{aligned}
\end{equation}
Now, note that the following remainder is well-defined on the whole probability space:
\[
\begin{aligned}
\mathfrak r_N
:={}&
\log\cZ_N(\beta,A_N)-\log\cZ_N(\beta,\bW_N)
+\frac12(L_N(A_N)-L_N(\bW_N))\\
&\quad
+\frac14\sum_{i=1}^N
\left[
(\widehat R_{N,ii}-1)^2-(R_{N,ii}-1)^2
\right].
\end{aligned}
\]
We have just shown that \(\mathfrak r_N = O_\P(N^{-2})\) on $\cE_N^\rho$.
Hence, for every \(\varepsilon>0\),
\[
\P\bigl(N|\mathfrak r_N|>\varepsilon\bigr)
\le
\P\bigl(N|\mathfrak r_N|>\varepsilon,~ \cE_N^\rho\bigr) + \P\bigl((\cE_N^\rho)^c\bigr)
\rightarrow0.
\]
Thus \(N\mathfrak r_N=o_{\P}(1)\), and \eqref{epfdl1} holds on the
whole probability space.

Now, by Proposition~\ref{prop:diagonal-corrections}~(iii)--(iv), we have:
\begin{eqnarray}\label{epfdl2}
-\frac N4\sum_{i=1}^N
\left[
(\widehat R_{N,ii}-1)^2
-(R_{N,ii}-1)^2
\right]
&\xrightarrow{\P}&
-\frac14\int_0^1
(h_\beta(x)^2-\bar h_\beta(x)^2)\d x\notag\\
&=&
-\frac{\beta^2}{2}\int_0^1\bar h_\beta(x)s_W(x)\d x
-\frac{\beta^4}{4}\int_0^1s_W(x)^2\d x,\notag\\
\end{eqnarray}
where we used the identity \(h_\beta=\bar h_\beta+\beta^2s_W\), which gives
\(
h_\beta^2-\bar h_\beta^2
=
2\beta^2\bar h_\beta s_W+\beta^4s_W^2.
\)
Theorem~\ref{thm:logdet} gives
\begin{equation}\label{epfdl3}
\begin{aligned}
&-\frac N2\Bigl(L_N(A_N)-L_N(\bW_N) +\frac{\beta^2}{N^2}\sum_{i<j}\omega_{ij}(1-\omega_{ij})\Bigr)\\
&\quad\xrightarrow{\mathrm d}
\cN\left(
\frac{\beta^2}{2}\int_0^1\bar h_\beta(x) s_W(x)\d x
+\frac{\beta^4}{4}\int_0^1s_W(x)^2\d x,\,
\tau_\beta^2
\right).
\end{aligned}
\end{equation}
Proposition \ref{thm:mainprop} now follows from \eqref{epfdl1}, \eqref{epfdl2} and \eqref{epfdl3}.
\end{proof}

We now complete the proof of Theorem \ref{thm:main} using Proposition \ref{thm:mainprop}. 
\begingroup
 To this end, for \(p\in[0,1]\) and
\(u\in\mathbb R\), define
\(
 \phi_p(u):=\log(1-p+pe^u).
\)
For \(t\ge0\), the even and odd parts of \(\phi_p\) in \(t\) are
\(\frac{\phi_p(t)+\phi_p(-t)}2\) and \(\frac{\phi_p(t)-\phi_p(-t)}2\), respectively.
We then have the following identity for \(s\in\{-1,1\}\):
\[
 1-p+pe^{ts}=\exp\left\{\frac{\phi_p(t)+\phi_p(-t)}2+s\cdot\frac{\phi_p(t)-\phi_p(-t)}2\right\}.
\]
Let us now define a symmetric hollow matrix \(\widetilde \bW_N\) by
\[
 \widetilde \omega_{ij}:=\frac{\phi_{\omega_{ij}}(\beta/N)-\phi_{\omega_{ij}}(-\beta/N)}{2(\beta/N)},
 \qquad i\ne j.
\]
By independence of the upper-triangular entries of \(A_N\), we have:
\begin{align}
 \E\cZ_N(\beta,A_N)
 &=2^{-N}\sum_{\bm{\sigma}\in\{-1,1\}^N}
 \prod_{i<j}
 \E\exp\left\{\frac{\beta}{N} A_{ij}\sigma_i\sigma_j\right\}\notag\\
 &=\exp\left\{\sum_{i<j}\frac{\phi_{\omega_{ij}}(\beta/N)+\phi_{\omega_{ij}}(-\beta/N)}2\right\}
 \cZ_N(\beta,\widetilde \bW_N).
 \label{eq:annealed-factorization}
\end{align}

By a Taylor expansion at \(t=0\), we have the following uniformly over \(p\in[0,1]\) and \(t\) in a fixed neighborhood of zero:

\[
 \frac{\phi_p(t)+\phi_p(-t)}2=\frac12p(1-p)t^2+O(t^4),
 \qquad
 \frac{\phi_p(t)-\phi_p(-t)}2 =pt+O(t^3).
\]
Consequently,
\begin{equation}\label{eq:annealed-prefactor}
\begin{aligned}
\sum_{i<j} \frac{\phi_{\omega_{ij}}(\beta/N)+\phi_{\omega_{ij}}(-\beta/N)}2
&= \frac{\beta^2}{2N^2}\sum_{i<j}\omega_{ij}(1-\omega_{ij})+O(N^{-2}),\\
\max_{i\ne j}|\widetilde \omega_{ij}-\omega_{ij}|&=O(N^{-2}).
\end{aligned}
\end{equation}
Moreover, \(0\leq \frac{1}{2}(\phi_p(t)-\phi_p(-t))\leq t\) for \(p\in[0,1]\) and \(t\geq0\), as
\[
 1\leq\frac{1-p+pe^t}{1-p+pe^{-t}}\leq e^{2t}.
\]
Thus the entries of \(\widetilde \bW_N\) belong to \([0,1]\). Next, set
\[
 B_N:=\frac{\beta}{N} \bW_N,
 \qquad
 \widetilde B_N:=\frac{\beta}{N} \widetilde \bW_N
 \qquad \text{and}\qquad
 D_N:=\widetilde B_N-B_N.
\]
Then \(D_N\) is hollow and it follows from \eqref{eq:annealed-prefactor} that
\[
 \max_{i,j}|D_{N,ij}|=O(N^{-3})
 \qquad \text{and} \qquad
 \|D_N\|_F+\|D_N\|_{\op}=O(N^{-2}).
\]
For \(u\in[0,1]\), let
\[
 B_N(u):=B_N+uD_N,
 \qquad
 R_N(u):=(I-B_N(u))^{-1}.
\]
Fix \(\rho\) such that \(\beta\|T_W\|_{\op}<\rho<1\), and choose any \(\rho'\in(\rho,1)\).
By Lemma~\ref{lem:high-temperature}~(i) and the bound
\(\|D_N\|_{\op}=O(N^{-2})\), we have for all sufficiently large \(N\),
\[
 \sup_{0\leq u\leq1}\|B_N(u)\|_{\op}\leq\rho'.
\]
Since \(\bW_N\) and \(\widetilde \bW_N\) are symmetric,
hollow, with entries in \([0,1]\), so is their convex combination
\((1-u)\bW_N+u\widetilde \bW_N\), and
\(B_N(u)=\frac{\beta}{N}\{(1-u)\bW_N+u\widetilde \bW_N\}\). Thus Proposition~\ref{prop:RG-comparison} applies to $B_N(u)$ uniformly in $u \in [0,1]$. 
Since \(\tr(D_N)=0\), we have:
\[
 \frac{\mathrm d}{\mathrm du}\log\det(I-B_N(u))
 =-\tr\{R_N(u)D_N\}
 =-\tr\{(R_N(u)-I)D_N\}.
\]
Since \(\|B_N(u)\|_F=O(1)\) and \(
\sup_{0\leq u\leq 1}\|R_N(u)\|_{\op}
\leq \frac{1}{1-\rho'}
\), we have: 
\[
 \|R_N(u)-I\|_F
 \leq \|R_N(u)\|_{\op}\|B_N(u)\|_F
 =O(1),
\]
and hence, by the Cauchy-Schwarz inequality,
\[
 |\tr\{(R_N(u)-I)D_N\}|
 \leq\|R_N(u)-I\|_F\|D_N\|_F
 =O(N^{-2}).
\]
Since \(B_N(0)=B_N\), \(B_N(1)=\widetilde B_N\), and
\[
\sup_{0\leq u\leq1}
\left|
\frac{\mathrm d}{\mathrm du}\log\det(I-B_N(u))
\right|
=O(N^{-2}),
\]
the fundamental theorem of calculus yields
\begin{equation}\label{eq:annealed-determinant-comparison}
\log\det(I-\widetilde B_N)-\log\det(I-B_N)
=
\int_0^1
\frac{\mathrm d}{\mathrm du}\log\det(I-B_N(u))\,\mathrm du
=
O(N^{-2}).   
\end{equation}
Let \(\bm{e}_i\) denote the \(i^\mathrm{th}\) standard basis vector of \(\mathbb R^N\).
The Neumann series, together with the facts that \(B_N(u)\) is hollow,
its entries are bounded by \(\beta/N\), and its operator norm is bounded
by \(\rho'\), gives
\[
\begin{aligned}
 \max_{1\leq i\leq N}\sup_{0\leq u\leq1}
 |R_N(u)_{ii}-1|
 &=
 \sup_{0\leq u\leq1}\max_{1\leq i\leq N}
 \left|
 \sum_{m=2}^{\infty}
 \bm e_i^{\top}B_N(u)^m\bm e_i
 \right|\\
 &\leq
 \sup_{0\leq u\leq1}\max_{1\leq i\leq N}
 \sum_{m=2}^{\infty}
 \|B_N(u)\bm e_i\|_2^2
 \|B_N(u)\|_{\op}^{m-2}\\
 &\leq
 \frac{\beta^2}{N}\sum_{m=2}^{\infty}(\rho')^{m-2}
 =\frac{\beta^2}{N(1-\rho')}
 =O(N^{-1}).
\end{aligned}
\]
Next, note that
\[
 \frac{\mathrm d}{\mathrm du}R_N(u)
 =R_N(u)D_NR_N(u),
\]
so
\[
\max_{1\leq i\leq N}\sup_{0\leq u\leq1}
 \left|\frac{\mathrm d}{\mathrm du}R_N(u)_{ii}\right|
 \leq
 \sup_{0\leq u\leq1}
 \|R_N(u)\|_{\op}^2\|D_N\|_{\op}
 =O(N^{-2}).
\]
Combining all these, we obtain the following:
\begin{eqnarray*}
&&
\sup_{0\leq u\leq1}
\left|
\frac{\mathrm d}{\mathrm du}
\sum_{i=1}^N\{R_N(u)_{ii}-1\}^2
\right|
\\
&=&
2\sup_{0\leq u\leq1}
\left|
\sum_{i=1}^N
\{R_N(u)_{ii}-1\}
\frac{\mathrm d}{\mathrm du}R_N(u)_{ii}
\right|
\\
&\leq&
2N
\left(
\max_{1\leq i\leq N}\sup_{0\leq u\leq1}
|R_N(u)_{ii}-1|
\right)
\left(
\max_{1\leq i\leq N}\sup_{0\leq u\leq1}
\left|
\frac{\mathrm d}{\mathrm du}R_N(u)_{ii}
\right|
\right)
\\
&=&
2N\cdot O(N^{-1})\cdot O(N^{-2})
=O(N^{-2}).
\end{eqnarray*}
It follows that
\begin{equation}
 \sum_{i=1}^N
 \left[(R_N(1)_{ii}-1)^2-(R_N(0)_{ii}-1)^2\right]
 =O(N^{-2}).
 \label{eq:annealed-diagonal-comparison}
\end{equation}
Applying Proposition~\ref{prop:RG-comparison} to \(B_N\) and
\(\widetilde B_N\), and using
\eqref{eq:annealed-determinant-comparison}, \eqref{eq:annealed-diagonal-comparison},
we obtain
\begin{eqnarray}
 \log\cZ_N(\beta,\widetilde \bW_N)
 -\log\cZ_N(\beta,\bW_N)
 &=&
 -\frac12\left[
 \log\det(I-\widetilde B_N)-\log\det(I-B_N)
 \right]
 \nonumber\\
 &-&
 \frac14\sum_{i=1}^N
 \left[
 (R_N(1)_{ii}-1)^2-(R_N(0)_{ii}-1)^2
 \right]
 +O(N^{-2})
 \nonumber\\
 &=&O(N^{-2}),
 \label{eq:annealed-effective-comparison}
\end{eqnarray}

Combining \eqref{eq:annealed-factorization},
\eqref{eq:annealed-prefactor}, and
\eqref{eq:annealed-effective-comparison}, we have:
\begin{equation*}
 \log\E\cZ_N(\beta,A_N)
 =\log\cZ_N(\beta,\bW_N)+\frac{\beta^2}{2N^2}\sum_{i<j}\omega_{ij}(1-\omega_{ij})+O(N^{-2}).
\end{equation*}
Therefore,
\begin{eqnarray*}
 &&N\left(
 \log\cZ_N(\beta,A_N)
 -\log\E\cZ_N(\beta,A_N)
 \right)\\
 &=&
 N\left(
 \log\cZ_N(\beta,A_N)-\log\cZ_N(\beta,\bW_N)-\frac{\beta^2}{2N^2}\sum_{i<j}\omega_{ij}(1-\omega_{ij})
 \right)+O(N^{-1}).
\end{eqnarray*}
Theorem \ref{thm:main} now follows from
Proposition \ref{thm:mainprop} and Slutsky's
theorem.
\endgroup

\section{Main ingredients for the proof of Theorem \ref{thm:main}}

The proof of Theorem \ref{thm:main} hinges on Proposition \ref{thm:mainprop}, which in turn relies heavily on the three results presented in this section.

\begin{ppn}
\label{prop:RG-comparison}
Let \(M\) be a symmetric hollow \(N\times N\) matrix with entries in
\([0,1]\), and set \(B=\beta M/N\).  Suppose that
\(\lambda_{\max}(B)\le1-\delta\) for some \(\delta>0\).  Then, $Z^G(B)=\det(I-B)^{-1/2}$. Further,
uniformly over all such matrices \(M\), we have:
\begin{equation}\label{logfbminusgb195462}
    \log \cZ_N(\beta,M)-\log Z^G(B)
=
-\frac14\sum_{i=1}^N
\left\{(I-B)^{-1}_{ii}-1\right\}^2
+O_{\beta,\delta}(N^{-2}).
\end{equation}
\end{ppn}

\begin{proof}
Since
\(\lambda_{\max}(B)\le1-\delta\), the matrix \(I-B\) is positive
definite. Hence, the Gaussian integral gives us:
\[Z^G(B)
=
(2\pi)^{-N/2}\int_{\mathbb R^N}
e^{-\bm{z}^{\top}(I-B)\bm{z}/2}\,\mathrm{d}\bm{z}
=
\det(I-B)^{-1/2}.
\]

Next, consider the successive replacement of the
Gaussian coordinates by Rademacher coordinates.  For
\(0\le i\le N\), let
\[
\bm Y_{(i)}
:=
(\sigma_1,\ldots,\sigma_i,Z_{i+1},\ldots,Z_N),
\]
so that \(\bm Y_{(0)}=\bm{Z}\sim \cN(\bm 0, I_N)\) and \(\bm Y_{(N)}=\bm{\sigma}\sim \mathrm{Unif}(\{-1,1\}^N)\).  Thus,
\[
\cZ_N(\beta,M)-Z^G(B)
=
\sum_{i=1}^N
\left[
\E \exp\left\{\frac{1}{2}\bm Y_{(i)}^{\top}B \bm Y_{(i)}\right\}-\E \exp\left\{\frac12 \bm Y_{(i-1)}^{\top}B \bm Y_{(i-1)}\right\}
\right].
\]
Let
\(\bm Y_{(i,0)}:=(\sigma_1,\ldots,\sigma_{i-1},0,Z_{i+1},\ldots,Z_N)\).
Since \(B_{ii}=0\),
\[
\exp\left\{\frac12(\bm Y_{(i,0)}+u\bm{e}_i)^{\top}B(\bm Y_{(i,0)}+u\bm{e}_i)\right\}
=
\exp\left\{\frac12\bm Y_{(i,0)}^\top B \bm Y_{(i,0)}\right\}e^{uL_i(\bm Y_{(i,0)})}.
\]
Conditioning on \(\bm Y_{(i,0)}\), we
therefore obtain
\[
\begin{aligned}
&\E_{\sigma_i}\exp\left\{\frac12(\bm Y_{(i,0)}+\sigma_i\bm{e}_i)^{\top}B(\bm Y_{(i,0)}+\sigma_i\bm{e}_i)\right\}
-\E_{Z_i}\exp\left\{\frac12(\bm Y_{(i,0)}+Z_i\bm{e}_i)^{\top}B(\bm Y_{(i,0)}+Z_i\bm{e}_i)\right\}\\
&=\exp\left\{\frac12\bm Y_{(i,0)}^\top B \bm Y_{(i,0)}\right\}(\cosh L_i(\bm Y_{(i,0)})-e^{L_i(\bm Y_{(i,0)})^2/2}).
\end{aligned}
\]

Define \(r(x):=\cosh x-e^{x^2/2}+x^4/12\). Taylor's theorem gives
\(|r(x)|\le K|x|^6\) for \(|x|\le1\).  On the other hand, for \(|x|>1\), using \(\cosh x\le e^{|x|} \le e^{1/2}e^{x^2/2}\), we obtain
\[
|r(x)|
\le e^{|x|}+e^{x^2/2}+\frac{x^4}{12}
\le K|x|^6e^{x^2/2}.
\]
  Thus,
\begin{equation}\label{eq:r-bound}
|r(x)|\lesssim |x|^6e^{x^2/2},
\qquad x\in\mathbb R.
\end{equation}
Summing the coordinate replacements now gives
\begin{equation}\label{eq:first-Lindeberg}
\cZ_N(\beta,M)-Z^G(B)
=
-\frac1{12}\sum_{i=1}^N
\E\left(\exp\left\{\frac12\bm Y_{(i,0)}^\top B \bm Y_{(i,0)}\right\} L_i(\bm Y_{(i,0)})^4\right)
+\mathcal R_N,
\end{equation}
where
\[
\mathcal R_N
:=
\sum_{i=1}^N
\E\left(\exp\left\{\frac12\bm Y_{(i,0)}^\top B \bm Y_{(i,0)}\right\}r(L_i(\bm Y_{(i,0)}))\right).
\]

Let $\bm b_i$ denote the $i^\mathrm{th}$ column of $B$. Since \(B=\beta M/N\) and \(0\le M_{ij}\le1\),
\(\|\bm{b}_i\|_2^2\le\beta^2/N\).  Hence, for all sufficiently large \(N\),
\(\|\bm{b}_i\|_2^2\le\delta/4\).  Applying
Lemma~\ref{lem:hybrid-domination}~(iii) with \(\bm{u}_0=\bm{u}_1=\bm{b}_i\) and \(m_1=6\),
together with \eqref{eq:r-bound}, gives
\[
|\mathcal R_N|
\lesssim_{\beta,\delta}
Z^G(B)\sum_{i=1}^N\|\bm{b}_i\|_2^6
\lesssim_{\beta,\delta}
Z^G(B) N^{-2}.
\]

By Lemma~\ref{lem:fourth-gaussianization}, the hybrid fourth-order
expectations in \eqref{eq:first-Lindeberg} may be replaced by their
Gaussian counterparts at a total cost
\(O_{\beta,\delta}(Z^G(B) N^{-2})\).  Let \(\bm Z_{(i,0)}\) be a standard
Gaussian vector with its \(i^\mathrm{th}\) coordinate set to zero.  Then, we have:
\begin{equation}\label{eq:Gaussianized-fourth-term}
\cZ_N(\beta,M)- Z^G(B)
=
-\frac1{12}\sum_{i=1}^N
\E\left(\exp\left\{\frac12\bm Z_{(i,0)}^\top B \bm Z_{(i,0)}\right\} L_i(\bm Z_{(i,0)})^4\right)
 +
 O_{\beta,\delta}(Z^G(B) N^{-2}).
\end{equation}

Let \(B^{(i)}\) be the principal submatrix obtained by deleting the
\(i\)-th row and column of \(B\), and let \(\bm{b}_i^{(-i)}\) be the
\(i\)-th column of \(B\) with its \(i\)-th entry removed.  Set
\[
R^{(i)}:=(I-B^{(i)})^{-1},
\qquad
q_i:=(\bm{b}_i^{(-i)})^{\top}R^{(i)}\bm{b}_i^{(-i)}.
\]
If \(\widetilde{\bm{Z}}^{(i)}\sim\cN(\bm{0},R^{(i)})\), then note that:
\[
\begin{aligned}
&\E\left(\exp\left\{\frac12\bm Z_{(i,0)}^\top B \bm Z_{(i,0)}\right\} L_i(\bm Z_{(i,0)})^4\right)\\
&\qquad=
\det(I-B^{(i)})^{-1/2}
\E\left[
\{(\bm{b}_i^{(-i)})^{\top}\widetilde{\bm{Z}}^{(i)}\}^4
\right]\\
&\qquad=
3q_i^2\det(I-B^{(i)})^{-1/2}.
\end{aligned}
\]
Moreover, the Schur complement identity gives
\[
\det(I-B)
=
\det(I-B^{(i)})(1-q_i).
\]
Substituting these identities into
\eqref{eq:Gaussianized-fourth-term}, we obtain
\begin{equation}\label{eq:ratio-expansion}
x_N := \frac{\cZ_N(\beta,M)-Z^G(B)}{Z^G(B)}
=
-\frac14\sum_{i=1}^Nq_i^2(1-q_i)^{1/2}
+
O_{\beta,\delta}(N^{-2}).
\end{equation}
Now, by eigenvalue interlacing, we can conclude that the smallest eigenvalue of \(I-B^{(i)}\)
is at least \(\delta\).  Hence,
\[
0\le q_i
\le
\delta^{-1}\|\bm{b}_i^{(-i)}\|_2^2
\le
\frac{\beta^2}{\delta N},
\qquad
\sum_{i=1}^Nq_i^3=O_{\beta,\delta}(N^{-2}).
\]
Since
\(\left|(1-q_i)^{1/2}-1\right|\le q_i\), we have from \eqref{eq:ratio-expansion},
\[
x_N
=
-\frac14\sum_{i=1}^Nq_i^2
+
O_{\beta,\delta}(N^{-2}),
\qquad
|x_N|=O_{\beta,\delta}(N^{-1}).
\]
Consequently,
\begin{equation}\label{eq:log-ratio-expansion}
\log \cZ_N(\beta,M) -\log Z^G(B)
=
\log(1+x_N)
=
-\frac14\sum_{i=1}^Nq_i^2
+
O_{\beta,\delta}(N^{-2}).
\end{equation}

Now, the matrix $I-B$ can be written in the following block form, with the \(i\)th coordinate listed first:
\[
I-B
=
\begin{pmatrix}
1 & -(\bm{b}_i^{(-i)})^{\top}\\
-\bm{b}_i^{(-i)} & I-B^{(i)}
\end{pmatrix}.
\]
The block inverse formula therefore gives
\((I-B)^{-1}_{ii}=(1-q_i)^{-1}\), and hence
\[
\left((I-B)^{-1}_{ii}-1\right)^2
=
\frac{q_i^2}{(1-q_i)^2}
=
q_i^2+O_{\beta,\delta}(q_i^3).
\]
Summing this over \(i\), and using
\(\sum_{i=1}^N q_i^3=O_{\beta,\delta}(N^{-2})\), \eqref{logfbminusgb195462}
follows from \eqref{eq:log-ratio-expansion}.
\end{proof}

\begin{ppn}\label{prop:diagonal-corrections}
Fix \(\rho\) such that \(\beta\|T_W\|_{\op}<\rho<1\), and let \(R_N\), \(\cE_N^\rho\) and \(\widehat R_N\) be as defined in \eqref{RNdef7224}, \eqref{epNrhodef} and \eqref{rnhatdef7842}, respectively. 
We have the following:
\begin{enumerate}
\item[(i)]\quad
\(
N^{-1}\sum_{i=1}^N
 \left[N(\widehat R_{N,ii}-1)-h_\beta(i/N)\right]^2
 \xrightarrow{\P}0.
\)
\item[(ii)] \quad
\(
 N^{-1}\sum_{i=1}^N
 \left[N(R_{N,ii}-1)-\bar h_\beta(i/N)\right]^2
 \rightarrow0.
\)
\item[(iii)] \quad
\(
 N\sum_{i=1}^N(\widehat R_{N,ii}-1)^2
 \xrightarrow{\P}\int_0^1h_\beta(x)^2\d x.
\)
\item[(iv)] \quad
\(
 N\sum_{i=1}^N(R_{N,ii}-1)^2
 \rightarrow\int_0^1\bar h_\beta(x)^2\d x.
\)
\end{enumerate}
\end{ppn}

\begin{proof}
For notational convenience, let us write \(A=A_N\).
On $\cE_N^{\rho}$, we have:
\begin{equation}\label{eq:random-resolvent-series}
 N(\widehat R_{N,ii}-1)
 =\sum_{m=2}^{\infty}\beta^m C_{m,N}(i),
 \qquad
 C_{m,N}(i):=N\left(\frac{A}{N}\right)^m_{ii}.
\end{equation}

To control the tail in \eqref{eq:random-resolvent-series}, fix $M\ge2$ and
put $B_N=\beta A/N$.  On $\cE_N^{\rho}$,
\[
 \left|N\sum_{m>M}\bm{e}_i^{\top}B_N^m\bm{e}_i\right|
 \le N\|B_N\bm{e}_i\|_2^2\sum_{m>M}\rho^{m-2}
 \le\frac{\beta^2\rho^{M-1}}{1-\rho},
\]
since $\|B_N\bm{e}_i\|_2^2\le\beta^2/N$.
Write
\[
 x_{i,N}:=N(\widehat R_{N,ii}-1).
\]
For every fixed \(M\ge2\), on \(\cE_N^\rho\), we have:
\[
 \begin{aligned}
 x_{i,N}-h_\beta(i/N)
 &=
 \beta^2(C_{2,N}(i)-d_W(i/N))
 +\sum_{m=3}^{M}\beta^m
 (C_{m,N}(i)-W_m(i/N,i/N))\\
 &\quad
 +r_{i,N}^{(M)}-\bar r_{i,N}^{(M)},
 \end{aligned}
\]
where
\[
 r_{i,N}^{(M)}
 :=
 \sum_{m>M}\beta^mC_{m,N}(i),
 \qquad
 \bar r_{i,N}^{(M)}
 :=
 \sum_{m>M}\beta^mW_m(i/N,i/N).
\]
This tail estimate and Lemma~\ref{lem:discrete-resolvent}~(iv) give,
uniformly in \(N\) and \(i\),
\[
 |r_{i,N}^{(M)}|
 \le
 \frac{\beta^2\rho^{M-1}}{1-\rho},
 \qquad
 |\bar r_{i,N}^{(M)}|
 \le
 \frac{\beta^2q^{M-1}}{1-q},
 \qquad
 q:=\beta\|T_W\|_{\op}<1.
\]
For the finite sum, using the Cauchy-Schwarz inequality
\[
 \left|\sum_{m=2}^{M}a_m\right|^2
 \le
 (M-1)\sum_{m=2}^{M}|a_m|^2,
\]
together with Lemma~\ref{lem:random-diagonal-powers}, we obtain the following
for every fixed \(M\),
\[
 \begin{aligned}
 &\E\left[
 \frac1N\sum_{i=1}^N
 \left|
 \beta^2(C_{2,N}(i)-d_W(i/N))
 +\sum_{m=3}^{M}\beta^m
 (C_{m,N}(i)-W_m(i/N,i/N))
 \right|^2
 \right]\\
 &\qquad\le
 (M-1)\Biggl\{
 \beta^4\frac1N\sum_{i=1}^N
 \E\left[
 \left(C_{2,N}(i)-d_W(i/N)\right)^2
 \right]\\
 &\qquad\quad+\sum_{m=3}^{M}\beta^{2m}
 \frac1N\sum_{i=1}^N
 \E\left[
 \left(C_{m,N}(i)-W_m(i/N,i/N)\right)^2
 \right]
 \Biggr\}
 \rightarrow0.
 \end{aligned}
\]
Part~\textnormal{(i)} follows by first letting \(N\to\infty\), then
\(M\to\infty\), and using \(\P((\cE_N^\rho)^c)\to0\) from
Lemma~\ref{lem:high-temperature}~(ii).
For the deterministic resolvent,
Lemma~\ref{lem:discrete-resolvent}~(v) gives
\[
\frac1N\sum_{i=1}^N
\left[N(R_{N,ii}-1)-\bar h_\beta(i/N)\right]^2
\rightarrow0,
\]
which proves part~\textnormal{(ii)}.

For part~\textnormal{(iii)}, set
\[
x_{i,N}:=N(\widehat R_{N,ii}-1),
\qquad
h_{i,N}:=h_\beta(i/N).
\]
By the Cauchy-Schwarz inequality, we have
\[
\begin{aligned}
\left|
\frac1N\sum_{i=1}^N x_{i,N}^2
-
\frac1N\sum_{i=1}^N h_{i,N}^2
\right|
\le
\left[
\frac1N\sum_{i=1}^N(x_{i,N}-h_{i,N})^2
\right]^{1/2}\cdot
\left[
\frac1N\sum_{i=1}^N(x_{i,N}+h_{i,N})^2
\right]^{1/2}.
\end{aligned}
\]
Moreover, by part~\textnormal{(i)} and the boundedness of \(h_\beta\), we have:
\[
\begin{aligned}
\left[
\frac1N\sum_{i=1}^N(x_{i,N}+h_{i,N})^2
\right]^{1/2}
&\le
\left[
\frac1N\sum_{i=1}^N(x_{i,N}-h_{i,N})^2
\right]^{1/2}+
2\left[
\frac1N\sum_{i=1}^N h_{i,N}^2
\right]^{1/2}
=O_{\P}(1).
\end{aligned}
\]
Consequently,
\[
\frac1N\sum_{i=1}^N x_{i,N}^2
-
\frac1N\sum_{i=1}^N h_\beta(i/N)^2
\xrightarrow{\P}0.
\]
Since \(h_\beta^2\) is Riemann integrable, part (iii) follows from the Riemann-sum convergence and the identity
\[
\frac1N\sum_{i=1}^N x_{i,N}^2
=
N\sum_{i=1}^N(\widehat R_{N,ii}-1)^2.
\]
Applying the same argument with
\(N(R_{N,ii}-1)\), \(\bar h_\beta(i/N)\), and part~\textnormal{(ii)}
proves part~\textnormal{(iv)}, thereby completing the proof of Proposition \ref{prop:diagonal-corrections}.
\end{proof}

\begin{thm}\label{thm:logdet}
 Suppose that $\beta>0$ satisfies $\beta\|T_W\|_{\op}<1$. 
Define
\[
 \mudet
 :=
 -\beta^2\int_0^1\bar h_\beta(x)s_W(x)\d x
 -\frac{\beta^4}{2}\int_0^1s_W(x)^2\d x.
\]
Then, as $N\rightarrow \infty$,
\[
 N\left\{
 L_N(A_N)-L_N(\bW_N)
 +\frac{\beta^2}{N^2}\sum_{i<j}\omega_{ij}(1-\omega_{ij})
 \right\}
 \xrightarrow{\mathrm d}
 \cN(\mudet,4\sbw^2).
\]
\end{thm}

\begin{proof}
 For ease of notation, let us write \(A=A_N\), \(\bW=\bW_N\), \(\bar A=\bar A_N\), and \(R=R_N\).  For all sufficiently large
\(N\), the matrix \(R=(I-\beta \bW/N)^{-1}\) is symmetric positive
definite and \(\|R\|_{\op}\) is uniformly bounded by
Lemma~\ref{lem:discrete-resolvent}~(i).
Fix \(\rho\) such that \(\beta\|T_W\|_{\op}<\rho<1\). Now, define
\[
\mathcal G_N
:=
\left\{
\left\|\frac{\beta A}{N}\right\|_{\op}\le\rho,\quad
\left\|\frac{\beta R\bar A}{N}\right\|_{\op}\le\frac12
\right\}.
\]
By Lemmas~\ref{lem:high-temperature}~(ii) and
\ref{lem:adjacency-norm}~(ii), together with the uniform bound on
\(\|R\|_{\op}\), we have \(\P(\mathcal G_N)\to1\).  Moreover,
\(\mathcal G_N\subseteq\cE_N^\rho\).  Since
\(R=(I-\beta \bW/N)^{-1}\), we have the identity
\[
I-\frac{\beta(\bW+\bar A)}{N}
=
\left(I-\frac{\beta \bW}{N}\right)
\left(I-\frac{\beta R\bar A}{N}\right),
\]
which gives
\[
L_N(A)-L_N(\bW)
=
\log\det\left(I-\frac{\beta R\bar A}{N}\right).
\]

Now, on \(\mathcal G_N\), let \(\lambda_1,\ldots,\lambda_N\) be the
eigenvalues of \(X:=\beta R\bar A/N\).  Since \(R\bar A\) is similar to the
symmetric matrix \(R^{1/2}\bar A R^{1/2}\) and \(\|X\|_{\op}\le1/2\), we have
\(|\lambda_i|\le1/2\).  Hence
\[
\begin{aligned}
\log\det(I-X)
=
\sum_{i=1}^N\log(1-\lambda_i)=
-\sum_{i=1}^N\sum_{k=1}^{\infty}\frac{\lambda_i^k}{k}=
-\sum_{k=1}^{\infty}\frac1k\sum_{i=1}^N\lambda_i^k=
-\sum_{k=1}^{\infty}\frac{\tr(X^k)}{k}.
\end{aligned}
\]
Therefore,
\[
L_N(A)-L_N(\bW)
=
-\sum_{k=1}^{\infty}
\frac{\beta^k}{kN^k}\tr\{(R\bar A)^k\}.
\]

Now, since \(\tr(\bar A)=0\) and \(\bar A\) is hollow, we have
\[
 \tr(R\bar A)
 =\tr\{(R-I)\bar A\}
 =2\sum_{i<j}R_{ij}\bar A_{ij},
\]
where in the last equality we used \(R_{ij}=(R-I)_{ij}\) for \(i\ne j\). Also,
\[
 \tr(R\bar A R\bar A)
 =
 \tr(\bar A^2)
 +2\tr\{(R-I)\bar A^2\}
 +\tr\{(R-I)\bar A(R-I)\bar A\}.
\]
Moreover, \(\tr(\bar A^2)=2\sum_{i<j}\bar A_{ij}^2\), and
\[
 \bar A_{ij}^2-\omega_{ij}(1-\omega_{ij})=(1-2\omega_{ij})\bar A_{ij}.
\]
It follows that on \(\mathcal G_N\),
\begin{align}
 &N\left\{
 L_N(A)-L_N(\bW)
 +\frac{\beta^2}{N^2}\sum_{i<j}\omega_{ij}(1-\omega_{ij})
 \right\}
 \nonumber\\
 &\quad=
 S_N
 -\frac{\beta^2}{N}\tr\{(R-I)\bar A^2\}
 -\frac{\beta^2}{2N}\tr\{(R-I)\bar A(R-I)\bar A\}
 \notag\\
 &\qquad
 -\frac{\beta^3}{3N^2}\tr\{(R\bar A)^3\}
 -\frac{\beta^4}{4N^3}\tr\{(R\bar A)^4\}
 +\mathcal T_N,
 \label{eq:logdet-decomposition}
\end{align}
where
\[
 S_N
 :=
 -\frac{\beta^2}{N}\sum_{i<j}a_{ij}^{(N)}\bar A_{ij},
 \qquad
 a_{ij}^{(N)}
 :=
 1-2\omega_{ij}+\frac{2N R_{ij}}{\beta},
\]
and
\[
 \mathcal T_N
 :=
 -N\sum_{k=5}^{\infty}
 \frac1k\tr\left\{\left(\frac{\beta R\bar A}{N}\right)^k\right\} \mathbbm{1}_{\mathcal{G}_N}.
\]

The summands of \(S_N\) are independent and centered, with
\[
 \operatorname{Var}(S_N)
 =
 \frac{\beta^4}{N^2}
 \sum_{i<j}(a_{ij}^{(N)})^2\omega_{ij}(1-\omega_{ij}).
\]
By Lemma~\ref{lem:high-temperature}~(i), the step-kernel with block
values \(\omega_{ij}\) converges in \(L^2([0,1]^2)\) to \(W\). By
Lemma~\ref{lem:discrete-resolvent}~(i)--(ii), 
\[
 \frac{N}{\beta}\bigl(R_{N,ij}-\delta_{ij}\bigr) \xrightarrow{L^2} U_\beta,
\] and these block values are
uniformly bounded. Since \(i<j\) in the definition of \(a_{ij}^{(N)}\),
we have \(R_{ij}=R_{N,ij}-\delta_{ij}\). Hence the step-kernel with block
values \(a_{ij}^{(N)}\) converges in \(L^2\) to
\(
 1-2W+2U_\beta=1+2(U_\beta-W).
\)
Similarly, the step-kernel with block values \(\omega_{ij}(1-\omega_{ij})\)
converges in \(L^2\) to \(W(1-W)\). Uniform boundedness therefore
implies that \(\omega_{ij}(1-\omega_{ij})(a_{ij}^{(N)})^2\)
converges in \(L^1\), under the step-kernel identification, to
\(
 W(1-W)\left[1+2(U_\beta-W)\right]^2.
\)

By symmetry, the omission of the diagonal blocks contributes
\(O(N^{-1})\) to \(\operatorname{Var}(S_N)\), and hence
\[
 \operatorname{Var}(S_N)
 \rightarrow
 4\sbw^2.
\]
Each summand of \(S_N\) is bounded in absolute value by \(K/N\), for some
\(K>0\), so Lindeberg's condition holds, and
\begin{equation}\label{eq:Sn-clt}
 S_N\xrightarrow{\mathrm d}
 \cN(0,4\sbw^2).
\end{equation}

For the \((R-I)\bar A^2\) term, set
\(s_{i,N}:=\sum_{j\ne i}\omega_{ij}(1-\omega_{ij})\). In the expansion of
\(\tr\{(R-I)\bar A^2\}\), the expectation of
\(\bar A_{jk}\bar A_{ki}\) vanishes unless \(i=j\), in which case it equals
\(\omega_{ik}(1-\omega_{ik})\). Therefore
\[
 \E\tr\{(R-I)\bar A^2\}
 =
 \sum_{i=1}^N (R_{ii}-1)s_{i,N}.
\]

Since \(W(1-W)\) is Riemann integrable,
Lemma~\ref{lem:riemann-approximation}~(ii), together with the fact that the contribution of the omitted diagonal terms in \(s_{i,N}\) is at most \(O(1)\), gives
\begin{equation}\label{eq:row-variance-profile}
\frac1N\sum_{i=1}^N
\left|\frac{s_{i,N}}{N}-s_W\left(\frac{i}{N}\right)\right|^2
\rightarrow0.
\end{equation}
Moreover, Lemma~\ref{lem:discrete-resolvent}~(i) gives
\[
 \sup_{N\ge N_0}\max_{1\le i\le N}N|R_{ii}-1|<\infty.
\]
On the other hand, the definition of \(\bar h_\beta\) and the bound
\(
 |W_m(x,x)|\le\|T_W\|_{\op}^{m-2}
\)
give
\[
 \sup_{x\in[0,1]}|\bar h_\beta(x)|
 \le
 \sum_{m=2}^{\infty}\beta^m\|T_W\|_{\op}^{m-2}
 =
 \frac{\beta^2}{1-\beta\|T_W\|_{\op}}.
\]
Finally, since \(0\le \omega_{ij}(1-\omega_{ij})\le1/4\),
\[
 0\le\frac{s_{i,N}}N\le\frac14,
 \qquad
 0\le s_W(x)\le\frac14.
\]
Hence, Lemma~\ref{lem:discrete-resolvent}~(v),
\eqref{eq:row-variance-profile}, the Cauchy--Schwarz inequality, and the
uniform bounds above give
\begin{align*}
&\left|
 \frac1N\sum_{i=1}^N N(R_{ii}-1)\frac{s_{i,N}}N
 -
 \frac1N\sum_{i=1}^N
 \bar h_\beta(i/N)s_W(i/N)
 \right|\\
&\quad\le
 \left\{
 \frac1N\sum_{i=1}^N
 \left|N(R_{ii}-1)-\bar h_\beta(i/N)\right|^2
 \right\}^{1/2}
 \left\{
 \frac1N\sum_{i=1}^N
 \left(\frac{s_{i,N}}N\right)^2
 \right\}^{1/2}\\
&\qquad+
 \left\{
 \frac1N\sum_{i=1}^N
 \bar h_\beta(i/N)^2
 \right\}^{1/2}
 \left\{
 \frac1N\sum_{i=1}^N
 \left|
 \frac{s_{i,N}}N-s_W(i/N)
 \right|^2
 \right\}^{1/2}
 \rightarrow0.
\end{align*}
Since \(\bar h_\beta s_W\) is Riemann integrable, it also follows that
\[
 \frac1N\sum_{i=1}^N
 \bar h_\beta(i/N)s_W(i/N)
 \rightarrow
 \int_0^1\bar h_\beta(x)s_W(x)\d x.
\]
Consequently,
\[
 \frac1N\E\tr\{(R-I)\bar A^2\}
 =
 \frac1N\sum_{i=1}^N N(R_{ii}-1)\frac{s_{i,N}}N
 \rightarrow
 \int_0^1\bar h_\beta(x)s_W(x)\d x.
\]
Also, by Lemma~\ref{lem:trace-variance}~(i),
\[
 \E\left[\left(
 \frac{\tr\{(R-I)\bar A^2\}-\E\tr\{(R-I)\bar A^2\}}{N}
 \right)^2\right]
 =
 O(N^{-1}).
\]
Hence,
\begin{equation}\label{eq:resolvent-quadratic-limit}
 -\frac{\beta^2}{N}\tr\{(R-I)\bar A^2\}
 \xrightarrow{\P}
 -\beta^2\int_0^1\bar h_\beta(x)s_W(x)\d x.
\end{equation}

Next, let \(\bm{e}_i\) denote the \(i^\mathrm{th}\) standard basis vector (vector of length $N$ with its $i^\mathrm{th}$ entry $1$ and all other entries $0$). Denote by $E_{ij}$ the $N\times N$ matrix with its $(i,j)^{\mathrm{th}}$ element $1$ and all other entries $0$, i.e.

\begin{equation}
 E_{ij}:=\bm{e}_i\bm{e}_j^{\top}.
    \label{def:eij}
\end{equation}

For the \((R-I)\bar A(R-I)\bar A\) term, put
\(X_{ij}:=E_{ij}+E_{ji}\) for \(i<j\). Since
\(\bar A=\sum_{i<j}\bar A_{ij}X_{ij}\), independence of
\(\{\bar A_{ij}\}_{i<j}\) gives
\[
\begin{aligned}
 \E\tr\{(R-I)\bar A(R-I)\bar A\}
 &=
 \sum_{i<j}\omega_{ij}(1-\omega_{ij})
 \tr\{(R-I)X_{ij}(R-I)X_{ij}\}\\
 &=
 2\sum_{i<j}\omega_{ij}(1-\omega_{ij})
 \left\{(R_{ii}-1)(R_{jj}-1)+R_{ij}^2\right\}
 =O(1),
\end{aligned}
\]
where the last bound follows from
\(\max_{i,j}|R_{ij}-\delta_{ij}|=O(N^{-1})\), by
Lemma~\ref{lem:discrete-resolvent}~(i). Together with
Lemma~\ref{lem:trace-variance}~(ii), this yields
\begin{equation}\label{eq:second-resolvent-negligible}
 \frac1N\tr\{(R-I)\bar A(R-I)\bar A\}
 \xrightarrow{\P}0.
\end{equation}

For the cubic term, the representation \(\bar A=\sum_{i<j}\bar{A}_{ij}X_{ij}\) gives
\[
\begin{aligned}
\E\tr\{(R\bar A)^3\}
&=
\sum_{i<j}\sum_{k<l}\sum_{r<s}
\E(\bar{A}_{ij}\bar{A}_{kl}\bar{A}_{rs})
\tr(RX_{ij}RX_{kl}RX_{rs})\\
&=
\sum_{i<j}\E(\bar{A}_{ij}^3)\tr\{(RX_{ij})^3\}.
\end{aligned}
\]
Indeed, unless the three unoriented edges coincide, at least one edge
variable appears exactly once, and independence and centering of $\bar{A}_{ij}$ make the
corresponding expectation zero.  

Next, since \(X_{ij}=\bm{e}_i\bm{e}_j^{\top}+\bm{e}_j\bm{e}_i^{\top}\),
\(RX_{ij}=UV^{\top}\), where \(U=(R\bm{e}_i,R\bm{e}_j)\) and
\(V=(\bm{e}_j,\bm{e}_i)\). Hence,
\[
\tr\{(RX_{ij})^3\}
=
\tr\{(V^{\top}U)^3\},
\qquad
V^{\top}U
=
\begin{pmatrix}
R_{ij} & R_{jj}\\
R_{ii} & R_{ij}
\end{pmatrix},
\]
and a direct \(2\times2\) computation gives
\[
\tr\{(RX_{ij})^3\}
=
2R_{ij}(R_{ij}^2+3R_{ii}R_{jj}).
\]
Lemma~\ref{lem:discrete-resolvent}~(i) gives
\(R_{ij}=O(N^{-1})\) for \(i\ne j\) and
\(R_{ii}=1+O(N^{-1})=O(1)\), uniformly in \(i,j\). Since
\(|\E(\bar{A}_{ij}^3)|\le1\), it follows that
\[
\left|\E\tr\{(R\bar A)^3\}\right|
=
O\left(\sum_{i<j}\frac1N\right)
=
O(N).
\]
Together with Lemma~\ref{lem:trace-variance}(iii), this yields
\begin{equation}\label{eq:cubic-negligible}
\frac1{N^2}\tr\{(R\bar A)^3\}
\xrightarrow{\P}0.
\end{equation}

For the fourth-order term,
\[
\tr(\bar A^4)
=
\sum_{i,j,k,\ell}
\bar{A}_{ij}\bar{A}_{jk}\bar{A}_{k\ell}\bar{A}_{\ell i}.
\]
Each summand in this expansion corresponds to the closed walk
\(i\to j\to k\to \ell \to i\).  Since the edge variables are independent
and centered, the corresponding expectation vanishes unless every unoriented edge
appearing in the walk is traversed at least twice.  As the walk has
length four and \(\bar A\) is hollow, there are only two possibilities.
First, a single edge \(\{i,j\}\) is traversed four times.  For each
\(i<j\), the two possible walks are
\(i\to j\to i\to j\to i\) and
\(j\to i\to j\to i\to j\), giving the contribution
\(2\E \bar{A}_{ij}^4\).  Second, two distinct edges are each traversed twice.  They must share a
vertex, say the edges are \(\{i,j\}\) and \(\{i,k\}\), where
\(j<k\) and \(j,k\ne i\).  The four corresponding rooted walks are
\[
i\to j\to i\to k\to i,\qquad
i\to k\to i\to j\to i,\qquad
j\to i\to k\to i\to j,\qquad
k\to i\to j\to i\to k.
\]
By independence, each contributes
\(\E(\bar{A}_{ij}^2\bar{A}_{ik}^2)=\omega_{ij}(1-\omega_{ij})\omega_{ik}(1-\omega_{ik})\).  Therefore
\[
\E\tr(\bar A^4)
=
2\sum_{i<j}\E \bar{A}_{ij}^4
+
4\sum_{i=1}^N
\sum_{\substack{j<k\\ j,k\ne i}}
\omega_{ij}(1-\omega_{ij})\omega_{ik}(1-\omega_{ik}).
\]
Since
\[
 2\sum_{i=1}^N s_{i,N}^2
 =
 4\sum_{i<j}\omega_{ij}^2(1-\omega_{ij})^2
 +
 4\sum_{i=1}^N
 \sum_{\substack{j<k\\ j,k\ne i}}\omega_{ij}\omega_{ik}(1-\omega_{ij})(1-\omega_{ik}),
\]
we obtain the exact identity
\[
 \E\tr(\bar A^4)
 =
 2\sum_{i=1}^N s_{i,N}^2
 +
 2\sum_{i<j}
 \left\{\E \bar{A}_{ij}^4-2\omega_{ij}^2(1-\omega_{ij})^2\right\}.
\]
The second sum is \(O(N^2)\), while
\eqref{eq:row-variance-profile}, boundedness, and
Riemann-sum convergence for the Riemann-integrable function \(s_W^2\) give
\[
 \frac{2}{N^3}\sum_{i=1}^N s_{i,N}^2
 =
 \frac2N\sum_{i=1}^N
 \left(\frac{s_{i,N}}{N}\right)^2
 \rightarrow
 2\int_0^1s_W(x)^2\d x.
\]
Consequently,
\[
 \frac1{N^3}\E\tr(\bar A^4)
 \rightarrow
 2\int_0^1s_W(x)^2\d x.
\]
The variance bound in Lemma~\ref{lem:trace-variance}~(iv), applied with
\(R=I\), gives
\(\operatorname{Var}\{\tr(\bar A^4)\}=O(N^4(\log N)^2)\).  Hence,
\[
 \frac1{N^3}\tr(\bar A^4)
 \xrightarrow{\P}
 2\int_0^1s_W(x)^2\d x.
\]

To replace \(\bar A\) by \(R\bar A\), set
\(Y:=\bar A/\sqrt N\). Then
\[
\frac1{N^3}
\left[
\tr\{(R\bar A)^4\}-\tr(\bar A^4)
\right]
=
\frac1N
\left[
\tr\{(RY)^4\}-\tr(Y^4)
\right].
\]
With the convention \((R-I)^0:=I\), the identity \(R=I+(R-I)\) gives
\[
\tr\{(RY)^4\}-\tr(Y^4)
=
\sum_{\substack{\bm{\varepsilon}\in\{0,1\}^4\\
\bm{\varepsilon}\ne\bm{0}}}
\tr\left(
(R-I)^{\varepsilon_1}Y
(R-I)^{\varepsilon_2}Y
(R-I)^{\varepsilon_3}Y
(R-I)^{\varepsilon_4}Y
\right).
\]
For a fixed \(\bm{\varepsilon}\in\{0,1\}^4\), let
\(m:=\varepsilon_1+\cdots+\varepsilon_4\ge1\). We may cyclically permute
the factors so that one copy of \(R-I\) occurs first. Let \(Q\) denote
the product between that copy of \(R-I\) and the final copy of \(Y\).
Thus \(Q\) contains \(m-1\) copies of \(R-I\) and three copies of \(Y\), and
\[
\operatorname{Tr}\left(
\prod_{i=1}^4 (R-I)^{\varepsilon_i}Y
\right)
=
\operatorname{Tr}\{(R-I)QY\}.
\]
By the Cauchy--Schwarz inequality, operator--Frobenius
submultiplicativity, and submultiplicativity of \(\|\cdot\|_{\op}\),
\[
\begin{aligned}
\left|\operatorname{Tr}\{(R-I)QY\}\right|
&\le \|R-I\|_F\,\|QY\|_F\\
&\le \|R-I\|_F\,\|Y\|_F\,\|Q\|_{\op}\\
&\le \|R-I\|_F\,\|Y\|_F\,
\|R-I\|_{\op}^{m-1}\|Y\|_{\op}^3.
\end{aligned}
\]
By Lemma~\ref{lem:discrete-resolvent}~(i),
\[
\|R-I\|_F^2
=
\sum_{i,j=1}^N|R_{ij}-\delta_{ij}|^2
=O(1),
\qquad
\|R-I\|_{\op}\le \|R\|_{\op}+1=O(1),
\]
while \(\|Y\|_F^2=N^{-1}\|\bar A\|_F^2\le N\). Since the expansion
contains only \(2^4-1=15\) terms,
\[
\frac1N
\left|
\tr\{(RY)^4\}-\tr(Y^4)
\right|
\le
\frac{K}{\sqrt N}\|Y\|_{\op}^3
\]
for some constant \(K>0\). By Lemma~\ref{lem:adjacency-norm}~(ii),
\(\|Y\|_{\op}=O_{\P}(\sqrt{\log N})\), and hence the last expression is
\(o_{\P}(1)\).
Consequently,
\[
\frac1{N^3}\tr\{(R\bar A)^4\}
=
\frac1{N^3}\tr(\bar A^4)+o_{\P}(1),
\]
and therefore
\begin{equation}\label{eq:fourth-limit}
-\frac{\beta^4}{4N^3}\tr\{(R\bar A)^4\}
\xrightarrow{\P}
-\frac{\beta^4}{2}\int_0^1s_W(x)^2\d x.
\end{equation}

For the remainder, the uniform bound on \(\|R\|_{\op}\) and
Lemma~\ref{lem:adjacency-norm}~(ii) give
\[
\Lambda_N := \left\|\frac{\beta R\bar A}{N}\right\|_{\op}
 =
 O_{\P}\left(\sqrt{\frac{\log N}{N}}\right).
\]
On $\mathcal{G}_N$, the bound
\(|\tr(\Xi^k)|\le N\|\Xi\|_{\op}^k\), valid for every
\(N\times N\) matrix \(\Xi\) and every integer \(k\ge1\), gives:
 \begin{equation}\label{eq:logdet-remainder}
 |\mathcal T_N|
 \le
 N^2\sum_{k=5}^{\infty}
\Lambda_N^k =  \frac{N^2\Lambda_N^5}{1-\Lambda_N}\le 
2N^2\Lambda_N^5
 =
 O_{\P}\left(\frac{(\log N)^{5/2}}{\sqrt N}\right)
 =
 o_\P(1).
 \end{equation}

Combining \eqref{eq:Sn-clt}, \eqref{eq:resolvent-quadratic-limit},
\eqref{eq:second-resolvent-negligible}, \eqref{eq:cubic-negligible},
\eqref{eq:fourth-limit} and \eqref{eq:logdet-remainder} with
\eqref{eq:logdet-decomposition} and \(\P(\mathcal G_N)\to1\)
proves Theorem~\ref{thm:logdet}.
\end{proof}

\vspace{0.1in}

\noindent\textbf{Acknowledgement} S. M. was supported by the  AcRF Tier 1 grant A-8002932-00-00.

\bibliographystyle{plain}
\bibliography{logpartbib1}
\appendix

\section{Some approximation results}
In this section, we provide several estimates that are relevant throughout the paper. We start with a basic result on Riemann approximation of scalar fields. 

\begin{lemma}\label{lem:riemann-approximation}
The following assertions hold.
\begin{enumerate}
\item[(i)] Let \(d\ge1\) and \(F:[0,1]^d\to\mathbb R\) be a Riemann integrable function.
Define
\[
 \widetilde F_N(x_1,\ldots,x_d)
 :=
 F\left(
 \frac{\lceil Nx_1\rceil}{N},
 \ldots,
 \frac{\lceil Nx_d\rceil}{N}
 \right).
\]
Then
\[
 \|\widetilde F_N-F\|_{L^2([0,1]^d)}\rightarrow0.
\]

\item[(ii)] Let \(r,s\ge1\) and
\(F:[0,1]^{r+s}\to\mathbb R\) be a Borel-measurable, Riemann integrable function. Define
\[
 G(\bm{x})
 :=\int_{[0,1]^s}F(\bm{x},\bm{y})\d\bm{y},
 \qquad \bm{x}\in[0,1]^r,
\]
where the integral is understood in the Lebesgue sense.
Then \(G\) is Riemann integrable and
\[
 \frac1{N^r}\sum_{\bm{i}\in[N]^r}
 \left|
 \frac1{N^s}\sum_{\bm{j}\in[N]^s}
 F\left(\frac{\bm{i}}N,\frac{\bm{j}}N\right)
 -G\left(\frac{\bm{i}}N\right)
 \right|^2
 \rightarrow 0.
\]

\item[(iii)] Let
\(
 [d]^2_{-}
 :=
 \{(a,b)\in[d]^2:a\ne b\}.
\)
If \(K:[0,1]^2\to\mathbb R\) is Riemann integrable, then, for every
\(k\ge1\) and every \(k\)-tuple of ordered pairs
\(
 \bigl((a_1,b_1),\ldots,(a_k,b_k)\bigr)
 \in\bigl([d]^2_{-}\bigr)^k,
\)
the map \(\iota:[0,1]^d\to\mathbb R\) defined by
\[
 \iota(x_1,\ldots,x_d)
 :=
 \prod_{\ell=1}^k K(x_{a_\ell},x_{b_\ell})
\]
is Riemann integrable on \([0,1]^d\).
\end{enumerate}
\end{lemma}

\begin{proof}
For every
\(\bm{i}=(i_1,\ldots,i_d)\in[N]^d\), define 
\[
 Q_{\bm{i},N} =  Q_{\bm{i},N}^{(d)}
 :=
 \prod_{r=1}^d
 \left[\frac{i_r-1}{N},\frac{i_r}{N}\right].
\]
Let us also define the oscillation of \(F\) on $Q_{\bm i, N}$ by
\[
\omega_{Q_{\bm{i},N}}(F)
 :=
 \sup_{\bm{x},\bm{y}\in Q_{\bm{i},N}}
 |F(\bm{x})-F(\bm{y})|
 =
 \sup_{\bm{x}\in Q_{\bm{i},N}}F(\bm{x})
 -
 \inf_{\bm{x}\in Q_{\bm{i},N}}F(\bm{x}).
\]
Since $F$ is Riemann integrable, we have:
\begin{equation}\label{DNosc88}
 \Delta_N(F) = \Delta_{N,d}(F)
 :=\frac1{N^d}\sum_{\bm{i}\in[N]^d}
\omega_{Q_{\bm{i},N}}(F)
 \rightarrow0,   
\end{equation}
since the right-hand side above is the difference between the upper and lower Darboux sums of $F$ over the regular $N$-grid on $[0,1]^d$. 

Now, note that on each $Q_{\bm i, N}$,  the difference \(|F-\widetilde F_N|\) is bounded by
\(\omega_{Q_{\bm{i},N}}(F)\).  By \eqref{DNosc88},
\[
 \|\widetilde F_N-F\|_{L^1} = \sum_{\bm i \in [N]^d} \int_{Q_{\bm i, N}} |F-\widetilde F_N|  \le \Delta_N(F).
\]
Since both functions are bounded by \(\|F\|_\infty\) in absolute
value,
\[
 \|\widetilde F_N-F\|_{L^2}^2
 \le2\|F\|_\infty\|\widetilde F_N-F\|_{L^1}\rightarrow0,
\]
which proves part~\textnormal{(i)}.

Next, for \(\bm{i}\in[N]^r\) and \(\bm{j}\in[N]^s\), define
\(
 Q_{(\bm{i},\bm{j}),N}
 :=
 Q_{\bm{i},N}^{(r)}\times Q_{\bm{j},N}^{(s)}.
\)
Then,
 \begin{eqnarray*}
\left| \frac1{N^s}\sum_{\bm{j}\in[N]^s}
 F\left(\frac{\bm{i}}N,\frac{\bm{j}}N\right)
 -G\left(\frac{\bm{i}}N\right)\right| &=&
 \left|\sum_{\bm{j}\in[N]^s}
 \int_{Q_{\bm{j},N}^{(s)}}
 \left[
 F\left(\frac{\bm{i}}N,\frac{\bm{j}}N\right)
 -
 F\left(\frac{\bm{i}}N,\bm{y}\right)
 \right]\d\bm{y}\right|\\&\le&  \sum_{\bm{j}\in[N]^s}
 \int_{Q_{\bm{j},N}^{(s)}}  \omega_{Q_{(\bm{i},\bm{j}),N}}(F)\d\bm{y}\\&=& \frac{1}{N^s}\sum_{\bm{j}\in[N]^s}
 \omega_{Q_{(\bm{i},\bm{j}),N}}(F)
 \end{eqnarray*}
Denote
\[
 \eta_{\bm{i},N}
 :=
 \frac1{N^s}\sum_{\bm{j}\in[N]^s}
 F\left(\frac{\bm{i}}N,\frac{\bm{j}}N\right)
 -G\left(\frac{\bm{i}}N\right).
\]
Averaging the preceding bound over \(\bm{i}\) gives
\[
 \frac1{N^r}\sum_{\bm{i}\in[N]^r}|\eta_{\bm{i},N}|
 \le
 \frac1{N^{r+s}}\sum_{(\bm{i},\bm{j})\in [N]^{r+s}}
 \omega_{Q_{(\bm{i},\bm{j}),N}}(F)
 =\Delta_{N,r+s}(F) \rightarrow 0.
\]
Since \(|\eta_{\bm{i},N}|\le2\|F\|_\infty\), it thus follows that
\[
 \frac1{N^r}\sum_{\bm{i}\in[N]^r}|\eta_{\bm{i},N}|^2
 \le2\|F\|_\infty\Delta_{N,r+s}(F)\rightarrow0.
\]

Moreover, for \(\bm{x},\bm{x}'\in Q_{\bm{i},N}^{(r)}\), 
\[
 |G(\bm{x})-G(\bm{x}')|
 \le
 \frac1{N^s}\sum_{\bm{j}\in[N]^s}
 \omega_{Q_{(\bm{i},\bm{j}),N}}(F).
\]
Hence, we have:
\[
 \frac1{N^r}\sum_{\bm{i}\in[N]^r}
 \omega_{Q_{\bm{i},N}^{(r)}}(G)
 =  \frac1{N^r}\sum_{\bm{i}\in[N]^r} \sup_{\bm{x},\bm{x}'\in Q_{\bm{i},N}^{(r)}}
\left|G(\bm{x})-G(\bm{x}')\right| \le \Delta_{N,r+s}(F)\rightarrow0.
\]
Since \(G\) is bounded, the Riemann criterion proves that \(G\) is
Riemann integrable. This proves part (ii).

For part (iii), since $K$ is Riemann integrable, the set \(D_K\) of discontinuity points of \(K\) has
two-dimensional Lebesgue measure zero. For each $(a,b)\in [d]_-^2$, define the projection map $\pi_{a,b}: [0,1]^d \to [0,1]^2$ as $\pi_{a,b}(\bm x) = (x_a,x_b)$. Note that if $K$ is continuous at $(x_{a_\ell}, x_{b_\ell})$ for all $1\le \ell\le k$, then $\iota$ is continuous at $\bm x$. Hence, the set of discontinuity points $D_\iota$ of $\iota$ satisfies:
$$D_\iota \subseteq \bigcup_{\ell=1}^k \pi_{a_\ell,b_\ell}^{-1}(D_K).$$
Since the $d$-dimensional Lebesgue measure of each $\pi_{a_\ell,b_\ell}^{-1}(D_K)$ is zero, $D_\iota$ has Lebesgue measure zero. Part~\textnormal{(iii)} now follows from the boundedness of $\iota$ and the 
Lebesgue criterion for Riemann integrability.
\end{proof}

\begin{lemma}
\label{lem:adjacency-norm}
  The following assertions hold.
  \begin{enumerate}
\item[(i)] There exists an absolute constant \(c>0\) such that
\[
 \P(\|\bar A_N\|_{\op}\ge t)
 \le 2N\exp\left\{-\frac{ct^2}{N+t}\right\},
 \qquad t\ge0.
\]
\item[(ii)] For every fixed \(q\ge1\),
\[
 \E\|\bar A_N\|_{\op}^q=O\{(N\log N)^{q/2}\}.
\]
In particular, \(\|\bar A\|_{\op}=O_{\P}(\sqrt{N\log N})\).
\end{enumerate}
\end{lemma}

\begin{proof}
Let $E_{ij}$ be as defined in \eqref{def:eij}. Write
\[
 \bar A=\sum_{i<j}\bar{A}_{ij}(E_{ij}+E_{ji}).
\]
The summands are independent, centered, self-adjoint, and have operator
norm at most one.  Moreover,
\[
 \sum_{i<j}
 \E\left[
 \bar{A}_{ij}^2(E_{ij}+E_{ji})^2
 \right]
 =
 \mathrm{diag}\left(\sum_{j\ne i}\omega_{ij}(1-\omega_{ij})\right)_{i=1}^N.
\]
Its operator norm is at most $N$.  Matrix Bernstein's inequality
\cite{Tropp2012} therefore gives
\[
 \P(\|\bar A\|_{\op}\ge t)
 \le 2N\exp\left\{-\frac{ct^2}{N+t}\right\}
 \qquad (t\ge0),
\]
for an absolute constant $c>0$, proving (i).

Next, let \(X_N:=\|\bar A\|_{\op}\).  Since \(|\bar{A}_{ij}|\le1\), we have
\(X_N\le N\) deterministically.  Moreover, for \(0\le t\le N\),
part~\textnormal{(i)} implies
\[
\P(X_N>t)
\le
2N\exp\left(-\frac{c_0t^2}{N}\right)
\]
for some \(c_0>0\).  Fix \(q\ge1\), and set
\(t_N=A\sqrt{N\log N}\), where \(A>0\) is sufficiently large.  Using
the tail-integral formula for moments followed by the change of variables
\(u=t/\sqrt N\), we obtain
\[
\begin{aligned}
\E X_N^q
&=
q\int_0^N t^{q-1}\P(X_N>t)\,\mathrm{d}t\\
&\le
t_N^q
+
2qN^{q/2+1}
\int_{A\sqrt{\log N}}^\infty
u^{q-1}e^{-c_0u^2}\,\mathrm{d}u\\
&\le A^q (N\log N)^{q/2} +  2qN^{q/2+1} \int_{A\sqrt{\log N}}^\infty
e^{-c_0u^2/2}\,\mathrm{d}u  \\
&=
O\bigl((N\log N)^{q/2}\bigr),
\end{aligned}
\]
where the last step follows from the standard Gaussian tail bound, upon
choosing \(A\) sufficiently large.  This proves part~\textnormal{(ii)}.
\end{proof}

\begin{lemma}
\label{lem:high-temperature}
Fix \(\rho\) such that \(\beta\|T_W\|_{\op}<\rho<1\). The following assertions hold.
\begin{enumerate}

\item[(i)] 
We have
\[
 \|W_N^{\mathrm{step}}-W\|_{L^2([0,1]^2)}
 \rightarrow0
\qquad \text{and}\qquad
 \left\|\frac{\bW_N}{N}\right\|_{\op}
 =\|T_N\|_{\op}
 \rightarrow\|T_W\|_{\op}.
\]
Consequently, there exists $N_0$ such that, for every $N\ge N_0$,
\[
 \beta\left\|\frac{\bW_N}{N}\right\|_{\op}\le\rho.
\]

\item[(ii)] 
There exist constants $c,C>0$ such that, for every $N\ge N_0$,
\[
 \P\left(
 \beta\left\|\frac{A_N}{N}\right\|_{\op}>\rho
 \right)
 \le CNe^{-cN}.
\]

\end{enumerate}
\end{lemma}

\begin{proof}
For ease of notation, in this proof, let us write \(\bW=\bW_N\), \(A=A_N\), and \(\bar A=\bar A_N=A-\bW\). 

\smallskip
\noindent
\emph{Proof of \textnormal{(i)}.}
To begin with, define the hollow step-kernel induced by $\bm W_N$ as
\begin{equation*}
W_N^{\mathrm{step}}(x,y)
:=\sum_{i,j=1}^N \omega_{ij}\,
\mathbbm{1}((x,y)\in I_{i,N}\times I_{j,N}),
\qquad (x,y)\in[0,1]^2,
\end{equation*}
where for each $1\le i\le N$, $I_{i,N}:=((i-1)/N,i/N]$. 
 Write $T_N:=T_{W_N^{\mathrm{step}}}$ for the corresponding integral operator.
Let $\widetilde W_N$ be the non-hollow step-kernel whose value on the
$(i,j)$ block is $W(i/N,j/N)$, including the case $i=j$. Lemma~\ref{lem:riemann-approximation}~(i) gives:
\[
 \|\widetilde W_N-W\|_{L^2([0,1]^2)}
 \rightarrow0.
\]
Note that the kernels $\widetilde W_N$ and $W_N^{\mathrm{step}}$ differ only on
the $N$ diagonal blocks, whose total area is $1/N$.  Since both kernels
take values in $[0,1]$, we have:
\[
 \|W_N^{\mathrm{step}}-\widetilde W_N\|_{L^2}^2
 \le\frac1N
\quad\implies \quad 
 \|W_N^{\mathrm{step}}-W\|_{L^2}
 \rightarrow0.
\]
Moreover, 
\[
 \begin{aligned}
 \left|
 \left\|\frac{\bW}{N}\right\|_{\op}-\|T_W\|_{\op}
 \right|
 =
 \bigl|\|T_N\|_{\op}-\|T_W\|_{\op}\bigr| \le \|T_N-T_W\|_{\op}\le \|T_N-T_W\|_{\HS}
 =\|W_N^{\mathrm{step}}-W\|_{L^2}.
 \end{aligned}
\]
Hence,
\(
 \left\|\bW/N\right\|_{\op}
 \rightarrow\|T_W\|_{\op}.
\)
Since
\(
 \beta\|T_W\|_{\op}<\rho,
\)
the last assertion in part~\textnormal{(i)} follows for all sufficiently
large \(N\).

\smallskip
\noindent
\emph{Proof of \textnormal{(ii)}.}
To begin with, define
\[
 \gamma:=\frac12\left(\rho-\beta\|T_W\|_{\op}\right)>0.
\]
By part~\textnormal{(i)}, we have the following for all $N$ sufficiently large:
\[
 \beta\left\|\frac{\bW}{N}\right\|_{\op}\le\rho-\gamma.
\]
Since \(A=\bW+\bar A\), we have:
\[
 \left\{\beta\left\|\frac AN\right\|_{\op}>\rho\right\}
 \subseteq
 \left\{\|\bar A\|_{\op}>\frac{\gamma N}{\beta}\right\}.
\]
Therefore, Lemma~\ref{lem:adjacency-norm}~(i), applied with
\(t=\gamma N/\beta\), proves part~\textnormal{(ii)}.
\end{proof}

\begin{lemma}
\label{lem:discrete-resolvent}
Write \(\bW=\bW_N\). Then the following assertions hold for all \(N\) sufficiently
large.
\begin{enumerate}
\item[(i)] The matrix $R_N$ satisfies
\[
 \sup_{N\ge N_0}\|R_N\|_{\op}\le\frac1{1-\rho},
 \qquad
 \sup_{N\ge N_0}N\max_{1\le i,j\le N}
 \left|R_{N,ij}-\delta_{ij}\right|\le K
\]
for some \(K=K(\beta,\rho)<\infty\).

\item[(ii)] Let \(\mathcal U_{\beta,N}\) be the kernel of
\(T_N(I-\beta T_N)^{-1}\). Equivalently, it is constant on each
\((i,j)\) block, with value
\[
 \mathcal U_{\beta,N}(x,y)
 =\frac{N}{\beta}\left(R_{N,ij}-\delta_{ij}\right)
\]
whenever
\(
 (x,y)\in
 ((i-1)/N,i/N]\times((j-1)/N,j/N].
\)
Then
\[
 \|\mathcal U_{\beta,N}-U_\beta\|_{L^2([0,1]^2)}\rightarrow0.
\]

\item[(iii)] For each fixed \(m\ge2\), define
\[
 D_{m,N}(i):=N\left(\frac{\bW}{N}\right)^m_{ii}.
\]
Then
\[
 \frac1N\sum_{i=1}^N
 \left|D_{m,N}(i)-W_m(i/N,i/N)\right|^2\rightarrow0.
\]

\item[(iv)] With \(q:=\beta\|T_W\|_{\op}<1\), for every \(M\ge2\),
\[
 \sup_{x\in[0,1]}
 \left|\sum_{m>M}\beta^mW_m(x,x)\right|
 \le\frac{\beta^2q^{M-1}}{1-q}.
\]

\item[(v)] The diagonal correction satisfies
\[
 \frac1N\sum_{i=1}^N
 \left|N(R_{N,ii}-1)-\bar h_\beta(i/N)\right|^2\rightarrow0.
\]
\end{enumerate}
\end{lemma}

\begin{proof}
\textnormal{(i).}
By Lemma~\ref{lem:high-temperature}~(i), for all sufficiently large \(N\),
\[
 \|R_N\|_{\op}
 \le
 \frac{1}{1-\beta\|\bW/N\|_{\op}}
 \le\frac{1}{1-\rho}.
\]
Moreover,
\[
 \frac{N}{\beta}(R_N-I)
 =
 \bW+\sum_{m=1}^{\infty}
 \beta^m\frac{\bW^{m+1}}{N^m}.
\]
Let \(\bm\omega_i\) denote the \(i\)th row of \(\bW\). Since
\(\|\bm\omega_i\|_2\le\sqrt N\), for every \(m\ge1\),
\[
 \begin{aligned}
 \left|
 \frac{(\bW^{m+1})_{ij}}{N^m}
 \right|
 &=
 \left|
 \left\langle
 \frac{\bm\omega_i}{\sqrt N},
 \left(\frac{\bW}{N}\right)^{m-1}
 \frac{\bm\omega_j}{\sqrt N}
 \right\rangle
 \right|\\
 &\le
 \left\|\frac{\bW}{N}\right\|_{\op}^{m-1}.
 \end{aligned}
\]
Therefore
\[
 \begin{aligned}
 \frac{N}{\beta}|R_{N,ij}-\delta_{ij}|
 &\le
 1+
 \sum_{m=1}^{\infty}
 \beta^m
 \left\|\frac{\bW}{N}\right\|_{\op}^{m-1}\\
 &\le
 1+\frac{\beta}{1-\rho}.
 \end{aligned}
\]
This proves part~\textnormal{(i)}.

\vspace{0.1in}
\noindent
\textnormal{(ii).}
For a bounded self-adjoint operator $Q$ such that
$\beta\|Q\|_{\op}<1$, define
\[
 \Phi_\beta(Q):=Q(I-\beta Q)^{-1}.
\]
For any two such operators \(S,T\), the resolvent
identity gives
\begin{equation}\label{eq:Phi-resolvent-identity}
 \Phi_\beta(S)-\Phi_\beta(T)
 =
 (I-\beta S)^{-1}
 (S-T)
 (I-\beta T)^{-1}.
\end{equation}

Now, under the block-constant identification, \(T_N\) is represented by
\(\bW/N\). Hence \(\Phi_\beta(T_N)\) is represented by
\[
 \frac{\bW}{N}
 \left(I-\frac{\beta\bW}{N}\right)^{-1}.
\]
Since a step-kernel with block values \(k_{ij}\) is represented in the
normalized block basis by the matrix \((k_{ij}/N)_{i,j}\), the block values
of the kernel of \(\Phi_\beta(T_N)\) are \(\bW R_N\). The following identity
\[
 R_N-I=\frac{\beta\bW}{N}R_N
\]
gives
\[
 \bW R_N=\frac{N}{\beta}(R_N-I).
\]
Thus the kernel of \(\Phi_\beta(T_N)\) is \(\mathcal U_{\beta,N}\).
Applying \eqref{eq:Phi-resolvent-identity} with \(S=T_N\) and \(T=T_W\),
\[
 \begin{aligned}
 \|\Phi_\beta(T_N)-\Phi_\beta(T_W)\|_{\HS}
 &\le
 \|(I-\beta T_N)^{-1}\|_{\op}\,
 \|T_N-T_W\|_{\HS}\,
 \|(I-\beta T_W)^{-1}\|_{\op}.
 \end{aligned}
\]
The two resolvent norms are uniformly bounded, and
Lemma~\ref{lem:high-temperature}~(i) gives
\(
 \|T_N-T_W\|_{\HS}\to0.
\)
Since \(\Phi_\beta(T_W)\) has kernel \(U_\beta\), part~\textnormal{(ii)} follows.

\vspace{0.1in}
\noindent
\textnormal{(iii).}
Fix \(m\ge2\). Writing \(v_0=v_m=i\), we have
\begin{equation}\label{eq:discrete-closed-walk}
 D_{m,N}(i)
 =
 \frac{1}{N^{m-1}}
 \sum_{v_1,\ldots,v_{m-1}=1}^N
 \prod_{r=0}^{m-1}\omega_{v_rv_{r+1}}.
\end{equation}
Let
\[
 \widetilde D_{m,N}(i)
 :=
 \frac{1}{N^{m-1}}
 \sum_{v_1,\ldots,v_{m-1}=1}^N
 \prod_{r=0}^{m-1}
 W(v_r/N,v_{r+1}/N),
\]
where again \(v_0=v_m=i\). Since
\(\omega_{uv}=W(u/N,v/N)\) whenever \(u\ne v\), a summand in
\eqref{eq:discrete-closed-walk} can differ from the corresponding summand in
\(\widetilde D_{m,N}(i)\) only if \(v_r=v_{r+1}\) for at least one
\(r\in\{0,\ldots,m-1\}\). For each fixed \(r\), this constraint leaves at
most \(m-2\) free indices and therefore gives at most \(N^{m-2}\) tuples.
Since all factors lie in \([0,1]\),
\begin{equation}\label{eq:hollow-diagonal-error}
 \max_{1\le i\le N}
 \left|D_{m,N}(i)-\widetilde D_{m,N}(i)\right|
 \le \frac{m}{N}.
\end{equation}
Define
\[
 F_m(x;x_1,\ldots,x_{m-1})
 :=
 W(x,x_1)W(x_1,x_2)\cdots W(x_{m-2},x_{m-1})W(x_{m-1},x).
\]
Then
\[
 W_m(x,x)
 =
 \int_{[0,1]^{m-1}}F_m(x;x_1,\ldots,x_{m-1})
 \,\d x_1\cdots\d x_{m-1},
\]
and
\[
 \widetilde D_{m,N}(i)
 =
 \frac{1}{N^{m-1}}
 \sum_{v_1,\ldots,v_{m-1}=1}^N
 F_m\left(
 \frac{i}{N};
 \frac{v_1}{N},\ldots,\frac{v_{m-1}}{N}
 \right).
\]
Since \(W\) is Riemann integrable,
Lemma~\ref{lem:riemann-approximation}~(iii) shows that \(F_m\) is Riemann
integrable on \([0,1]^m\). Applying
Lemma~\ref{lem:riemann-approximation}~(ii) with \(r=1\) and \(s=m-1\) gives
\[
 \frac1N\sum_{i=1}^N
 \left|
 \widetilde D_{m,N}(i)-W_m(i/N,i/N)
 \right|^2
 \rightarrow0.
\]
Together with \eqref{eq:hollow-diagonal-error}, this proves part~\textnormal{(iii)}.

\vspace{0.1in}
\noindent
\textnormal{(iv).}
For every \(m\ge2\), symmetry of \(W\) gives
\[
 W_m(x,x)
 =
 \left\langle
 W(x,\cdot),
 T_W^{m-2}W(\cdot,x)
 \right\rangle_{L^2}.
\]
Consequently,
\[
 |W_m(x,x)|
 \le
 \|W(x,\cdot)\|_2\,
 \|T_W\|_{\op}^{m-2}\,
 \|W(\cdot,x)\|_2
 \le
 \|T_W\|_{\op}^{m-2}.
\]
If \(q:=\beta\|T_W\|_{\op}<1\), then
\[
 \sup_{x\in[0,1]}
 \left|
 \sum_{m>M}\beta^mW_m(x,x)
 \right|
 \le
 \beta^2\sum_{m>M}q^{m-2}
 =
 \frac{\beta^2q^{M-1}}{1-q}.
\]

\vspace{0.1in}
\noindent
\textnormal{(v).}
The Neumann series gives
\[
 R_N
 =
 \sum_{m=0}^{\infty}
 \left(\frac{\beta\bW}{N}\right)^m.
\]
Since \(\omega_{ii}=0\),
\[
 N(R_{N,ii}-1)
 =
 \sum_{m=2}^{\infty}\beta^mD_{m,N}(i).
\]
The series converges absolutely because
\(\beta\|\bW/N\|_{\op}\le\rho<1\) for all \(N\ge N_0\).

Set \(B_N:=\beta\bW/N\). Since \(B_N\) is symmetric, for every \(m\ge2\),
\[
 \bm e_i^{\top}B_N^m\bm e_i
 =
 (B_N\bm e_i)^{\top}B_N^{m-2}(B_N\bm e_i).
\]
Hence, for all large \(N\),
\[
 \begin{aligned}
 \left|
 N\sum_{m>M}\bm e_i^{\top}B_N^m\bm e_i
 \right|
 &\le
 N\|B_N\bm e_i\|_2^2
 \sum_{m>M}\rho^{m-2}.
 \end{aligned}
\]
Moreover,
\[
 \|B_N\bm e_i\|_2^2
 =
 \frac{\beta^2}{N^2}\sum_{j=1}^N\omega_{ij}^2
 \le \frac{\beta^2}{N}.
\]
Therefore
\[
 \max_{1\le i\le N}
 \left|
 N\sum_{m>M}\bm e_i^{\top}B_N^m\bm e_i
 \right|
 \le
 \frac{\beta^2\rho^{M-1}}{1-\rho}.
\]

For a vector \(\bm a\in\mathbb R^N\), write
\[
 \|\bm a\|_{2,N}:=
 \left(\frac1N\sum_{i=1}^N|a_i|^2\right)^{1/2}.
\]
Define
\[
 \varepsilon_N
 :=
 \left\|
 \left(N(R_{N,ii}-1)-\bar h_\beta(i/N)\right)_{i=1}^N
 \right\|_{2,N},
\]
and
\[
 \delta_{m,N}
 :=
 \left\|
 \left(D_{m,N}(i)-W_m(i/N,i/N)\right)_{i=1}^N
 \right\|_{2,N}.
\]
Splitting the two series at \(M\) and applying Minkowski's inequality,
\[
 \begin{aligned}
 \varepsilon_N
 \le
 \sum_{m=2}^{M}\beta^m\delta_{m,N}
 &+\left\|
 \left(\sum_{m>M}\beta^mD_{m,N}(i)\right)_{i=1}^N
 \right\|_{2,N}\\
 &+\left\|
 \left(\sum_{m>M}\beta^mW_m(i/N,i/N)\right)_{i=1}^N
 \right\|_{2,N}.
 \end{aligned}
\]
Since \(\|\bm a\|_{2,N}\le\max_i|a_i|\), the discrete and continuum tail
bounds give
\[
 \varepsilon_N
 \le
 \sum_{m=2}^{M}\beta^m\delta_{m,N}
 +\frac{\beta^2\rho^{M-1}}{1-\rho}
 +\frac{\beta^2q^{M-1}}{1-q}.
\]
For every fixed \(M\), part~\textnormal{(iii)} implies that the first term
converges to zero as \(N\to\infty\). Thus
\[
 \limsup_{N\to\infty}\varepsilon_N
 \le
 \frac{\beta^2\rho^{M-1}}{1-\rho}
 +\frac{\beta^2q^{M-1}}{1-q}.
\]
Letting \(M\to\infty\) proves part~\textnormal{(v)}.
\end{proof}

\begin{lemma}
\label{lem:random-diagonal-powers}
For \(m\ge2\), define
\[
 C_{m,N}(i):=N\left(\frac{A_N}{N}\right)^m_{ii}.
\]
Then the following assertions hold.
\begin{enumerate}
\item[(i)] As $N\rightarrow \infty$,
\[
 \frac1N\sum_{i=1}^N
 \E\left[\left(C_{2,N}(i)-d_W(i/N)\right)^2\right]\rightarrow0.
\]
\item[(ii)] For every fixed \(m\ge3\), as $N\rightarrow \infty$,
\[
 \frac1N\sum_{i=1}^N
 \E\left[\left(C_{m,N}(i)-W_m(i/N,i/N)\right)^2\right]\rightarrow0.
\]
\end{enumerate}
\end{lemma}

\begin{proof}
(i). For $m=2$, we have:
\[
 C_{2,N}(i)=\frac1N\sum_{j=1}^NA_{ij}.
\]
Moreover, \(\mathrm{Var}(C_{2,N}(i))=O(N^{-1})\) uniformly in \(i\).
Applying Lemma~\ref{lem:riemann-approximation}~(ii) to \(F=W\), and
accounting for the omitted diagonal term, gives
\[
 \frac1N\sum_{i=1}^N
 \left|\E C_{2,N}(i)-d_W\left(\frac{i}{N}\right)\right|^2
 \rightarrow0.
\]
The bias--variance decomposition now gives us:
\[
 \frac1N\sum_{i=1}^N
 \E\left[\left(C_{2,N}(i)-d_W(i/N)\right)^2\right] \rightarrow0.
\]
This proves part~\textnormal{(i)}.

\vspace{0.1in}
\noindent
(ii). Fix $m\ge3$.  Writing $v_0=v_m=i$, note that
\begin{equation}\label{eq:random-walk-expansion}
 C_{m,N}(i)
 =N^{1-m}\sum_{1\le v_1,\ldots,v_{m-1}\le N}
 \prod_{r=0}^{m-1}A_{v_rv_{r+1}}.
\end{equation}
Let \(D_{m,N}(i)\) be as in Lemma~\ref{lem:discrete-resolvent}~(iii).
Note that \(D_{m,N}(i)\) is obtained from \eqref{eq:random-walk-expansion} by
replacing every $A_{uv}$ with $\omega_{uv}$.  The expectation of a walk
monomial differs from the corresponding
product of edge means only if an unoriented edge is traversed at least twice.
For fixed $m$ and root $i$, there are $O(N^{m-2})$ such walks. Also,
self-loops contribute zero to both $C_{m,N}(i)$ and $D_{m,N}(i)$.
Therefore, there exists a constant $K_m$ such that
\begin{equation}\label{eq:walk-bias}
 \max_i\left|\E C_{m,N}(i)-D_{m,N}(i)\right|
 \le\frac{K_m}{N}.
\end{equation}

In view of Lemma~\ref{lem:discrete-resolvent}~(iii) and
\eqref{eq:walk-bias}, we have:
\begin{equation}\label{eq:walk-expectation}
 \frac1N\sum_{i=1}^N
 \left|\E C_{m,N}(i)-W_m(i/N,i/N)\right|^2\rightarrow0.
\end{equation}

Next, let $X_\gamma$ be the monomial in
\eqref{eq:random-walk-expansion} associated with a rooted closed walk
\(\gamma\).  Two monomials are independent when their unoriented edge sets
are disjoint.  The number of ordered pairs of rooted length-$m$ walks
sharing an edge is $O_m(N^{2m-3})$.  Since $|X_\gamma|\le1$, there exists
a constant $K_m'$ such that
\begin{equation}\label{eq:walk-variance}
 \max_i\operatorname{Var}(C_{m,N}(i))\le\frac{K_m'}{N}.
\end{equation}
Combining \eqref{eq:walk-expectation} and \eqref{eq:walk-variance}, for each fixed $m\ge3$, we have:
\[
 \frac1N\sum_{i=1}^N
 \E\left[\left(C_{m,N}(i)-W_m(i/N,i/N)\right)^2\right] \rightarrow0.
\]
This proves part~\textnormal{(ii)}.
\end{proof}

\begin{lemma}\label{lem:trace-variance}
Let \(\bar A=\bar A_N\) and \(R=R_N\). Then:
\begin{enumerate}
\item[(i)] \(\operatorname{Var}\{\tr((R-I)\bar A^2)\}=O(N)\).
\item[(ii)] \(\operatorname{Var}\{\tr((R-I)\bar A(R-I)\bar A)\}=O(1)\).
\item[(iii)] \(\operatorname{Var}\{\tr((R\bar A)^3)\}=O(N^3\log N)\).
\item[(iv)] \(\operatorname{Var}\{\tr((R\bar A)^4)\}
=O(N^4(\log N)^2)\).
\end{enumerate}
Moreover, the bounds in parts~\textnormal{(iii)}--\textnormal{(iv)}
remain valid when \(R\) is replaced by any deterministic, symmetric,
positive-definite matrix with uniformly bounded operator norm.
\end{lemma}

\begin{proof}
Let $\delta_{i,j}:= \mathbbm{1}_{i= j}$ denote the Kronecker delta.  By Lemma~\ref{lem:discrete-resolvent}~(i), there exists
\(K<\infty\), independent of \(N\), such that
\[
\max_{1\le i,j\le N}|R_{ij}-\delta_{ij}|\le \frac{K}{N}.
\]
We can write
\[
\tr\{(R-I)\bar A^2\}
=
\sum_{i,j,k}(R_{ij}-\delta_{ij})X_{ijk},
\qquad
X_{ijk}:=\bar A_{jk}\bar A_{ki}.
\]
Therefore
\[
\operatorname{Var}\{\tr((R-I)\bar A^2)\}
=
\sum_{i,j,k}\sum_{i',j',k'}
(R_{ij}-\delta_{ij})(R_{i'j'}-\delta_{i'j'})
\operatorname{Cov}(X_{ijk},X_{i'j'k'}).
\]
The monomial \(X_{ijk}\) involves the unoriented edges
\(\{j,k\}\) and \(\{k,i\}\), counted with multiplicity. By independence
and centering, the covariance can be nonzero only if the two monomials
involve the same unordered edge multiset. For each fixed \((i,j,k)\), only
\(O(1)\) triplets \((i',j',k')\) satisfy this condition. Moreover,
\(|\operatorname{Cov}(X_{ijk},X_{i'j'k'})|\le2\). Hence
\[
\operatorname{Var}\{\tr((R-I)\bar A^2)\}
=O\left(N^3\cdot N^{-2}\right)
=O(N),
\]
which proves part~\textnormal{(i)}.

Similarly,
\[
\tr\{(R-I)\bar A(R-I)\bar A\}
=
\sum_{i,j,k,\ell}
(R_{ij}-\delta_{ij})(R_{k\ell}-\delta_{k\ell})Y_{ijk\ell},
\qquad
Y_{ijk\ell}:=\bar A_{jk}\bar A_{\ell i}.
\]
Terms with \(j=k\) or \(\ell=i\) vanish. For the remaining terms,
associate with \(Y_{ijk\ell}\) the multiset of unoriented edges
\[
\mathcal M_{ijk\ell}
:=
\bigl\{\{j,k\},\{\ell,i\}\bigr\},
\]
with multiplicities retained. Independence and centering imply that
\(
\operatorname{Cov}(Y_{ijk\ell},Y_{i'j'k'\ell'})=0
\)
unless
\(
\mathcal M_{ijk\ell}=\mathcal M_{i'j'k'\ell'}.
\)
For each fixed \((i,j,k,\ell)\), there are only \(O(1)\) quadruples
\((i',j',k',\ell')\) having the same edge multiset, and all covariance
terms are uniformly bounded. The product of the four resolvent-difference
entries appearing in each variance summand is \(O(N^{-4})\). Since there
are \(O(N^4)\) choices of \((i,j,k,\ell)\),
\[
\operatorname{Var}\{\tr((R-I)\bar A(R-I)\bar A)\}=O(1),
\]
which proves part~\textnormal{(ii)}.

For \(k\in\{3,4\}\), set \(S:=R^{1/2}\bar A R^{1/2}\).  Since
\(R\bar A=R^{1/2}SR^{-1/2}\), we have
\(\tr((R\bar A)^k)=\tr(S^k)\).  For each \(a<b\), let \(\bar{A}_{ab}'\) be an independent copy of
\(\bar{A}_{ab}\), and let \(\bar A^{(ab)}\) be obtained from \(\bar A\) by replacing
both \(\bar{A}_{ab}\) and \(\bar{A}_{ba}\) by \(\bar{A}_{ab}'\).  Then
\[
\bar A^{(ab)}
=
\bar A+\xi_{ab}(E_{ab}+E_{ba}),
\qquad
S^{(ab)} :=
R^{1/2}\bar A^{(ab)}R^{1/2}
=
S+D_{ab},
\]
where
\[
\xi_{ab}:=\bar{A}_{ab}'-\bar{A}_{ab},
\qquad
D_{ab}:=
\xi_{ab}R^{1/2}(E_{ab}+E_{ba})R^{1/2}.
\]
Define
\(\Delta_{ab}:=\tr((S^{(ab)})^k)-\tr(S^k)\).  Since the variables
\((\bar{A}_{ab})_{a<b}\) are independent, the
Efron--Stein inequality (see Theorem 3.1 in \cite{BoucheronLugosiMassart2013}) gives
\[
\operatorname{Var}\{\tr(S^k)\}
\le
\frac12\sum_{a<b}\E \Delta_{ab}^2.
\]

Since \(\bar{A}_{ab}'\) and \(\bar{A}_{ab}\) are centered Bernoulli edge variables with the same centering,
we have \(|\xi_{ab}|\le1\).  Moreover,
\(\operatorname{rank}(D_{ab})\le2\) and
\(\|D_{ab}\|_{\op}\le K\), where \(K\) is independent of \(N,a,b\).
In the expansion of
\(\tr((S+D_{ab})^k)-\tr(S^k)\), the terms containing exactly one
factor \(D_{ab}\) contribute \(k\tr(D_{ab}S^{k-1})\), while every other term contains \(\ell\ge2\) factors of \(D_{ab}\),
and has the form
\[
\tr\left(
D_{ab}S^{r_1}D_{ab}S^{r_2}\cdots D_{ab}S^{r_\ell}
\right),
\qquad
r_1+\cdots+r_\ell=k-\ell.
\]
The matrix inside the trace has rank at most two.  Hence
\[
\begin{aligned}
\left|
\tr\left(
D_{ab}S^{r_1}\cdots D_{ab}S^{r_\ell}
\right)
\right|
&\le
2\|D_{ab}\|_{\op}^{\ell}\|S\|_{\op}^{k-\ell}\\
&\le
K_k\{1+\|S\|_{\op}^{k-2}\},
\end{aligned}
\]
for some constant $K_k$ independent of $N$.  Since the number of such
terms depends only on \(k\),
\[
\Delta_{ab}
=
k\tr(D_{ab}S^{k-1})+Q_{ab},
\qquad
|Q_{ab}|
\le
K_k\{1+\|S\|_{\op}^{k-2}\}.
\]

It remains to sum the squared contributions of these two terms. Set
\(M:=R^{1/2}S^{k-1}R^{1/2}\).  Since \(M\) is symmetric,
\[
\tr(D_{ab}S^{k-1})
=
\xi_{ab}\tr\{(E_{ab}+E_{ba})M\}
=
2\xi_{ab}M_{ab}.
\]
Using \(|\xi_{ab}|\le1\), we obtain
\[
\begin{aligned}
\sum_{a<b}|\tr(D_{ab}S^{k-1})|^2
\le
4\sum_{a<b}M_{ab}^2
\le
2\|M\|_F^2\le
K\|S\|_{\op}^{2k-4}\|S\|_F^2,
\end{aligned}
\]
where in the last inequality above, we used
\(\|M\|_F\le\|R\|_{\op}\|S^{k-1}\|_F
\le K\|S\|_{\op}^{k-2}\|S\|_F\) for some constant $K$.
Furthermore, since there are less than $N^2$ pairs \(a<b\), we have
\[
\sum_{a<b}|Q_{ab}|^2
\le
2K_k^2N^2\left(1+\|S\|_{\op}^{2k-4}\right).
\]
Combining these estimates with
\((u+v)^2\le 2 u^2+2 v^2\) in the Efron--Stein inequality, we have:
\begin{equation}\label{eq:efron-bound}
\operatorname{Var}\{\tr(S^k)\}
\le
C_k\E\left[
\|S\|_{\op}^{2k-4}\|S\|_F^2
+
N^2\left(1+\|S\|_{\op}^{2k-4}\right)
\right]
\end{equation}
for some constant $C_k$ independent of $N$.
By Lemmas~\ref{lem:adjacency-norm}~(ii) and
\ref{lem:discrete-resolvent}~(i), we have the following for every fixed \(q\),
\[
 \E\|S\|_{\op}^q=O\{(N\log N)^{q/2}\}.
\]

Moreover, \(\|S\|_F^2\lesssim_\rho \|\bar A\|_F^2\le N^2\).
Substituting these estimates into \eqref{eq:efron-bound}, first with
\(k=3\) and then with \(k=4\), gives
\[
\operatorname{Var}\{\tr((R\bar A)^3)\}
=
O(N^3\log N),
\qquad
\operatorname{Var}\{\tr((R\bar A)^4)\}
=
O\bigl(N^4(\log N)^2\bigr).
\]
This proves parts~\textnormal{(iii)} and \textnormal{(iv)}.  

For the final
assertion, let \(R\) be any deterministic, symmetric, positive-definite
matrix with uniformly bounded operator norm, and set
\(S=R^{1/2}\bar A R^{1/2}\). Then
\[
 \|S\|_{\op}
 \le \|R\|_{\op}\|\bar A\|_{\op},
 \qquad
 \|S\|_F
 \le \|R\|_{\op}\|\bar A\|_F.
\]
Consequently, Lemma~\ref{lem:adjacency-norm}~(ii) gives the same
moment bounds for \(\|S\|_{\op}\), while
\(\|S\|_F^2=O(N^2)\). The proof now follows unchanged from the preceding argument.
\end{proof}

\section{Rademacher-Gaussian estimates}
In this section, we establish two auxiliary estimates for hybrid random
vectors with Rademacher and Gaussian components.  These estimates are used
in the proof of Proposition~\ref{prop:RG-comparison} to control the remainder
in the coordinatewise Lindeberg replacement and to replace the resulting
fourth-order hybrid terms by their fully Gaussian counterparts.

\begin{lemma}\label{lem:hybrid-domination}
Let \(B\) be a symmetric \(N\times N\) matrix with zero diagonal, and
suppose that \(\lambda_{\max}(B)\le1-\delta\) for some \(\delta>0\).
Let \(\bm{X}:=(X_1,\ldots,X_N)\) have independent coordinates, where, for
each \(i\), \(X_i\) is either a Rademacher random variable, a standard
Gaussian random variable, or identically zero; the choice may depend
on \(i\).  Then the following statements
hold.

\begin{enumerate}
\item[(i)] If \(S\) is nonnegative definite, has rank at most \(r\), and
\[
\left\|(I-B)^{-1/2}S(I-B)^{-1/2}\right\|_{\op}\le\frac12,
\]
then
\[
\E\exp\left\{\frac12\bm{X}^{\top}(B+S)\bm{X}\right\}
\le
2^{r/2}\det(I-B)^{-1/2}.
\]

\item[(ii)] Fix \(q\ge1\) and nonnegative integers
\(m_1,\ldots,m_q\).
There exists \(K=K(\delta,q,m_1,\ldots,m_q)<\infty\) such that, for all
deterministic vectors \(\bm{u}_1,\ldots,\bm{u}_q\),
\[
\E\left[
e^{\bm{X}^{\top}B\bm{X}/2}
\prod_{j=1}^q|\bm{u}_j^{\top}\bm{X}|^{m_j}
\right]
\le
K\det(I-B)^{-1/2}
\prod_{j=1}^q\|\bm{u}_j\|_2^{m_j}.
\]

\item[(iii)] Under the assumptions of \textnormal{(ii)}, if
\(\bm{u}_0\in\mathbb R^N\) satisfies \(\|\bm{u}_0\|_2^2\le\delta/4\), then there exists 
$K<\infty$ such that, for all
deterministic vectors \(\bm{u}_1,\ldots,\bm{u}_q\),
\[
\E\left[
e^{\bm{X}^{\top}B\bm{X}/2+(\bm{u}_0^{\top}\bm{X})^2/2}
\prod_{j=1}^q|\bm{u}_j^{\top}\bm{X}|^{m_j}
\right]
\le
K\det(I-B)^{-1/2}
\prod_{j=1}^q\|\bm{u}_j\|_2^{m_j}.
\]
\end{enumerate}
\end{lemma}

\begin{proof}
Let \(Q\) be
symmetric, with \(\lambda_{\max}(Q)<1\) and \(Q_{jj}\ge0\) for every
\(j\), and let \(\bm{G}\) have independent standard Gaussian coordinates.
We claim that
\begin{equation}\label{eq:coordinate-domination}
\E e^{\bm{X}^{\top}Q\bm{X}/2}
\le
\E e^{\bm{G}^{\top}Q\bm{G}/2}
=
\det(I-Q)^{-1/2}.
\end{equation}
To see \eqref{eq:coordinate-domination}, note that with
$\bm{X}_{-j}:=(X_i)_{i\ne j}$, \(a:=Q_{jj}\in[0,1)\), and
\(\ell:=\sum_{k\ne j}Q_{jk}X_k\), we have
\[
\E\left[e^{\bm{X}^{\top}Q\bm{X}/2}\Big| \bm{X}_{-j}\right]
=g(\bm{X}_{-j})\E_{X_j} e^{aX_j^2/2+\ell X_j},
\]
where $g(\bm{t}):=\exp(\bm{t}^{\top}\widetilde Q_j\bm{t}/2)$ and $\widetilde Q_j$ is the
$(N-1)\times(N-1)$ submatrix of $Q$ obtained by deleting its $j^\mathrm{th}$ row
and column. Now, if a random variable \(\xi\) is Rademacher, then
\begin{equation}\label{eq:one-coordinate-domination}
\E_\xi e^{a\xi^2/2+\ell\xi}
=
e^{a/2}\cosh\ell
\le
(1-a)^{-1/2}
\exp\left\{\frac{\ell^2}{2(1-a)}\right\}.
\end{equation}
Here we used \(\cosh\ell\le e^{\ell^2/2}\),
\(e^{a/2}\le(1-a)^{-1/2}\), and \(1/(1-a)\ge1\).  The same inequality
also holds when \(\xi\equiv0\), since the right-hand side is at least
one.  The right-hand side of \eqref{eq:one-coordinate-domination} is precisely
\(\E_{G_j}e^{aG_j^2/2+\ell G_j}\). Sequentially replacing each coordinate
of $\bm{X}$ by the corresponding coordinate of $\bm{G}$, and applying the conditional
argument each time, gives \eqref{eq:coordinate-domination}.

For part~\textnormal{(i)}, since
\(\lambda_{\max}(B)\le1-\delta\), the matrix \(I-B\) is positive
definite.  Set
\[
\widetilde S:=(I-B)^{-1/2}S(I-B)^{-1/2}.
\]
Then \(\widetilde S\) is nonnegative definite, has rank at most \(r\),
and all its eigenvalues lie in \([0,1/2]\).  Moreover,
\[
I-B-S
=
(I-B)^{1/2}(I-\widetilde S)(I-B)^{1/2},
\]
so \(I-B-S\) is positive definite.  Since \(B\) has zero diagonal and
\(S\) is nonnegative definite, \(B+S\) has nonnegative diagonal.
Applying \eqref{eq:coordinate-domination} with \(Q=B+S\) gives
\[
\begin{aligned}
\E\exp\left\{\frac12\bm{X}^{\top}(B+S)\bm{X}\right\}
&\le
\det(I-B-S)^{-1/2}\\
&=
\det(I-B)^{-1/2}\det(I-\widetilde S)^{-1/2}.
\end{aligned}
\]
Since at most \(r\) eigenvalues of \(\widetilde S\) are nonzero and each is
at most \(1/2\), we have
\(\det(I-\widetilde S)^{-1/2}\le2^{r/2}\),  proving
part~\textnormal{(i)}.

For part~\textnormal{(ii)}, let
\(\mathcal J:=\{j:m_j>0\}\).  If \(\bm{u}_j=\bm{0}\) for some
\(j\in\mathcal J\), both sides of the assertion vanish, so assume
otherwise and set
\(\widehat{\bm{u}}_j:=\bm{u}_j/\|\bm{u}_j\|_2\). Now, note that for every \(m\ge1\) and
\(\varepsilon>0\), there exists \(c_{m,\varepsilon}<\infty\) such that
\(|x|^m\le c_{m,\varepsilon}e^{\varepsilon x^2/2}\) for all $x\in \mathbb{R}$.  Therefore
\begin{equation}\label{eq:polynomial-perturbation}
\prod_{j=1}^q|\bm{u}_j^{\top}\bm{X}|^{m_j}
\le
K_\varepsilon
\prod_{j=1}^q\|\bm{u}_j\|_2^{m_j}
\exp\left\{\frac12\bm{X}^{\top}S_\varepsilon\bm{X}\right\},
\end{equation}
where
\(S_\varepsilon:=\varepsilon
\sum_{j\in\mathcal J}\widehat{\bm{u}}_j\widehat{\bm{u}}_j^{\top}\).
The smallest eigenvalue of \(I-B\) is at least \(\delta\), and hence
\[
\left\|
(I-B)^{-1/2}S_\varepsilon(I-B)^{-1/2}
\right\|_{\op}
\le
\frac{\varepsilon|\mathcal J|}{\delta}
\le
\frac{\varepsilon q}{\delta}.
\]
Choosing \(\varepsilon=\delta/(2q)\), part \textnormal{(i)} applies to
\(S_\varepsilon\), whose rank is at most \(q\). This, combined with 
\eqref{eq:polynomial-perturbation}, gives part~\textnormal{(ii)}.

For part~\textnormal{(iii)}, take
\[
S
:=
\bm{u}_0\bm{u}_0^{\top}
+
\varepsilon
\sum_{j\in\mathcal J}
\widehat{\bm{u}}_j\widehat{\bm{u}}_j^{\top}.
\]
The first term satisfies
\[
\left\|
(I-B)^{-1/2}\bm{u}_0\bm{u}_0^{\top}(I-B)^{-1/2}
\right\|_{\op}
=
\|(I-B)^{-1/2}\bm{u}_0\|_2^2
\le
\frac{\|\bm{u}_0\|_2^2}{\delta}
\le
\frac14.
\]
Taking \(\varepsilon=\delta/(4q)\) makes the transformed operator norm
of the second term at most \(1/4\).  Thus the transformed norm of \(S\)
is at most \(1/2\), and its rank is at most \(q+1\).  Applying
part~\textnormal{(i)} proves part~\textnormal{(iii)} and completes the
proof of Lemma \ref{lem:hybrid-domination}.
\end{proof}

\begin{lemma}
\label{lem:fourth-gaussianization}
Let \(M\) be a deterministic symmetric hollow matrix with entries in
\([0,1]\). Set \(B:=\beta M/N\), and suppose that
\(\lambda_{\max}(B)\le1-\delta\) for some \(\delta>0\).  Let
\((\sigma_j)_{j\le N}\) be independent Rademacher variables and
\((Z_j)_{j\le N}\) be independent standard Gaussian variables, 
the two families being independent.  For \(1\le i\le N\), recall that,
\[
\bm Y_{(i,0)}
:=
(\sigma_1,\ldots,\sigma_{i-1},0,Z_{i+1},\ldots,Z_N),
\qquad
\bm Z_{(i,0)}
:=
(Z_1,\ldots,Z_{i-1},0,Z_{i+1},\ldots,Z_N).
\]
Then there exists \(K_{\beta,\delta}<\infty\), independent of \(N\)
and \(M\), such that

\begin{eqnarray*}
&&\sum_{i=1}^N\left|
\E\left(\exp\left\{\frac12\bm Y_{(i,0)}^\top B \bm Y_{(i,0)}\right\} L_i(\bm Y_{(i,0)})^4\right) -
\E\left(\exp\left\{\frac12 \bm Z_{(i,0)}^{\top}B \bm Z_{(i,0)}\right\} L_i(\bm Z_{(i,0)})^4\right)\right|\\ &\le& K_{\beta,\delta} Z^G(B) N^{-2}.
\end{eqnarray*}

\end{lemma}

\begin{proof}
For fixed \(i\), we replace the Rademacher coordinates
\(1,\ldots,i-1\) successively by Gaussians.  More precisely, for
\(0\le k\le i-1\), let
\[
\bm Y_{(i,k)}
:=
(Z_1,\ldots,Z_k,\sigma_{k+1},\ldots,\sigma_{i-1},
0,Z_{i+1},\ldots,Z_N).
\]
Thus, \(\bm Y_{(i,i-1)}= \bm Z_{(i,0)}\). Writing
\(F_i(\bm{x}):=\exp(\bm{x}^{\top}B \bm{x}/2) L_i(\bm{x})^4\), we have by a telescoping argument
\[
\E F_i(\bm Y_{(i,0)})-\E F_i(\bm Z_{(i,0)})
=
\sum_{k=1}^{i-1}\Delta_{ik},
\qquad
\Delta_{ik}
:=
\E F_i(\bm Y_{(i,k-1)})-\E F_i(\bm Y_{(i,k)}).
\]
Fix \(k<i\), and define
\[
\bm{X}_0
:=
(Z_1,\ldots,Z_{k-1},0,\sigma_{k+1},\ldots,\sigma_{i-1},
0,Z_{i+1},\ldots,Z_N).
\]
Thus,
\[
\bm Y_{(i,k-1)}=\bm{X}_0+\sigma_k\bm{e}_k
\qquad \text{and}\qquad
\bm Y_{(i,k)}=\bm{X}_0+Z_k\bm{e}_k.
\]

Set
\(a:=L_i(\bm{X}_0)\), \(c:=L_k(\bm{X}_0)\), and \(d:=B_{ik}\).  Since
\(B_{kk}=0\), we have
\[
F_i(\bm{X}_0+u\bm{e}_k)
=
\exp\left(\frac12\bm{X}_0^{\top}B \bm{X}_0\right) e^{cu}(a+du)^4
=:\psi(u).
\]
Now, it is straightforward to verify that for every integer \(m\ge0\),
\begin{equation}\label{eq:psi-derivative}
\psi^{(m)}(u)
=
\exp\left(\frac12\bm{X}_0^{\top}B \bm{X}_0\right) e^{cu}
\sum_{\ell=0}^{\min\{m,4\}}
\binom m\ell
\frac{4!}{(4-\ell)!}
d^\ell c^{m-\ell}(a+du)^{4-\ell}.
\end{equation}
A fifth-order Taylor expansion, together with the facts
\(\E\sigma_k^r=\E Z_k^r\) for \(r=0,1,2,3,5\) and
\(\E\sigma_k^4-\E Z_k^4=-2\), gives
\begin{equation}\label{eq:taylor-replacement}
\Delta_{ik}
=
-\frac1{12}\E\psi^{(4)}(0)+\eta_{ik},
\end{equation}
where
\[
\eta_{ik}
=
\frac1{5!}\int_0^1(1-s)^5
\E\left[
\sigma_k^6\psi^{(6)}(s\sigma_k)
-
Z_k^6\psi^{(6)}(sZ_k)
\right]\d s.
\]

Let us first bound the fourth-order term.  Since \(B=\beta M/N\) and
\(0\le M_{rs}\le1\), we have:
\begin{equation}\label{eq:column-bounds}
|d|\le\frac{\beta}{N}
\qquad \text{and}\qquad
\|\bm{b}_j\|_2^2
=
\sum_{r=1}^N B_{rj}^2
\le\frac{\beta^2}{N}.
\end{equation}
Setting \(m=4\) and \(u=0\) in \eqref{eq:psi-derivative}, and recalling
that \(a=\bm{b}_i^{\top}\bm{X}_0\) and \(c=\bm{b}_k^{\top}\bm{X}_0\), we obtain
\[
\begin{aligned}
|\E\psi^{(4)}(0)|
&\lesssim
\sum_{\ell=0}^4 |d|^\ell
\E\left[
\exp\left(\frac12\bm{X}_0^{\top}B \bm{X}_0\right)
|\bm{b}_i^{\top}\bm{X}_0|^{4-\ell}
|\bm{b}_k^{\top}\bm{X}_0|^{4-\ell}
\right]\\
&\lesssim_{\beta,\delta}
Z^G(B)
\sum_{\ell=0}^4
|d|^\ell\|\bm{b}_i\|_2^{4-\ell}\|\bm{b}_k\|_2^{4-\ell},
\end{aligned}
\]
where the second inequality follows from
Lemma~\ref{lem:hybrid-domination}~(ii).
By \eqref{eq:column-bounds}, each summand is bounded by
\[
\left(\frac{\beta}{N}\right)^\ell
\left(\frac{\beta}{\sqrt N}\right)^{2(4-\ell)}
=
O_\beta(N^{-4}).
\]
Hence, we have:
\[
|\E\psi^{(4)}(0)|
\lesssim_{\beta,\delta} Z^G(B) N^{-4}.
\]

We next control the remainder.  Now, it is quite straightforward to see that for either
\(\zeta=\sigma_k\) or \(\zeta=Z_k\), every fixed integer \(m\ge0\),
and \(0\le s\le1\), 
\begin{equation}\label{eq:gaussian-linear-moment}
\E_\zeta\left[|\zeta|^m e^{sc\zeta}\right]
\lesssim_m
(1+|c|^m)e^{c^2/2}.
\end{equation}
Using \eqref{eq:psi-derivative} with \(m=6\), for
\(\zeta\in\{\sigma_k,Z_k\}\) and \(0\le s\le1\), we obtain
\[
|\zeta^6\psi^{(6)}(s\zeta)|
\lesssim
\exp\left(\frac12\bm{X}_0^{\top}B \bm{X}_0\right) e^{sc\zeta}
\sum_{\ell=0}^4
|d|^\ell |c|^{6-\ell}|\zeta|^6
|a+ds\zeta|^{4-\ell}.
\]
Expanding the last factor and using \(s\le1\) gives
\[
|a+ds\zeta|^{4-\ell}
\lesssim
\sum_{r=0}^{4-\ell}
|a|^{4-\ell-r}|d|^r|\zeta|^r.
\]
Hence, conditionally on \(\bm{X}_0\), we have:
\[
\begin{aligned}
\E_\zeta|\zeta^6\psi^{(6)}(s\zeta)|
\lesssim
\exp\left(\frac12\bm{X}_0^{\top}B \bm{X}_0\right)
\sum_{\ell=0}^4\sum_{r=0}^{4-\ell}
|d|^{\ell+r}|a|^{4-\ell-r}|c|^{6-\ell} 
\E_\zeta\!\left[|\zeta|^{6+r}e^{sc\zeta}\right].
\end{aligned}
\]
Recall that by \eqref{eq:gaussian-linear-moment},
\[
\E_\zeta\!\left[|\zeta|^{6+r}e^{sc\zeta}\right]
\lesssim
(1+|c|^{6+r})e^{c^2/2}.
\]
Now, the definition of \(\eta_{ik}\), followed by the triangle inequality
gives:
\[
\begin{aligned}
|\eta_{ik}|
\lesssim
\sum_{\ell=0}^4\sum_{r=0}^{4-\ell}
|d|^{\ell+r}
\E\Big[
\exp\left(\frac12\bm{X}_0^{\top}B \bm{X}_0\right) e^{c^2/2}|a|^{4-\ell-r}\times
\bigl\{|c|^{6-\ell}+|c|^{12+r-\ell}\bigr\}
\Big].
\end{aligned}
\]

Now, note that by \eqref{eq:column-bounds}, for all sufficiently large \(N\),
\(\|\bm{b}_k\|_2^2\le\delta/4\).  Since
\(a=\bm{b}_i^{\top}\bm{X}_0\) and \(c=\bm{b}_k^{\top}\bm{X}_0\),
Lemma~\ref{lem:hybrid-domination}~(iii) gives
\[
\begin{aligned}
|\eta_{ik}|
\lesssim_{\beta,\delta}
Z^G(B)
\sum_{\ell=0}^4\sum_{r=0}^{4-\ell}
|d|^{\ell+r}\|\bm{b}_i\|_2^{4-\ell-r}
\left(
\|\bm{b}_k\|_2^{6-\ell}
+
\|\bm{b}_k\|_2^{12+r-\ell}
\right).
\end{aligned}
\]
By \eqref{eq:column-bounds}, we have:
\[
|d|^{\ell+r}\|\bm{b}_i\|_2^{4-\ell-r}
\|\bm{b}_k\|_2^{6-\ell}
=
O_{\beta}(N^{-5-r/2})\quad\text{and} \quad |d|^{\ell+r}\|\bm{b}_i\|_2^{4-\ell-r}
\|\bm{b}_k\|_2^{12+r-\ell}
=
O_{\beta}(N^{-8-r}).
\]
Consequently,
\[
|\eta_{ik}|
\lesssim_{\beta,\delta}
Z^G(B)
 N^{-5}.
\]
This, together with \eqref{eq:taylor-replacement}, gives
\[
|\Delta_{ik}|
\lesssim_{\beta,\delta} Z^G(B) N^{-4}.
\]
Summing the telescoping bounds over \(k<i\) and then over \(i\), we have:
\[
\begin{aligned}
\sum_{i=1}^N
\left|
\E F_i(\bm Y_{(i,0)})-\E F_i(\bm Z_{(i,0)})
\right|
&\le
\sum_{i=1}^N\sum_{k=1}^{i-1}|\Delta_{ik}|\\
&\lesssim_{\beta,\delta}
Z^G(B) N^{-4}
\frac{N(N-1)}2\\
&\lesssim_{\beta,\delta}
Z^G(B) N^{-2}.
\end{aligned}
\]
This proves the assertion for all sufficiently large \(N\).

Finally, for the first finitely many terms \(N\le N_0\) (say), we apply
Lemma~\ref{lem:hybrid-domination}~(ii) with \(q=1\), \(m_1=4\), and
\(\bm u_1= \bm b_i\), separately to \(\bm Y_{(i,0)}\) and \(\bm Z_{(i,0)}\).  Using
\(\|\bm b_i\|_2^2\le\beta^2/N\), we obtain
\begin{eqnarray*}
&&\sum_{i=1}^N\left|
\E\left(\exp\left\{\frac12\bm Y_{(i,0)}^\top B \bm Y_{(i,0)}\right\} L_i(\bm Y_{(i,0)})^4\right) -
\E\left(\exp\left\{\frac12\bm Z_{(i,0)}^\top B \bm Z_{(i,0)}\right\} L_i(\bm Z_{(i,0)})^4\right)\right|\\
&\lesssim_\delta&
Z^G(B)\sum_{i=1}^N\|\bm{b}_i\|_2^4
\lesssim_\delta
 \beta^4 Z^G(B) N^{-1}
\lesssim_{\beta,\delta}
Z^G(B) N_0 N^{-2}.
\end{eqnarray*}
The proof of Lemma \ref{lem:fourth-gaussianization} is now complete.
\end{proof}

\end{document}